\documentclass{ws-m3as}

\usepackage{amsmath}
\usepackage{amssymb}
\usepackage{t1enc}
\usepackage[english]{babel}
\usepackage{verbatim} 
\usepackage{graphicx}
\usepackage[dvips]{epsfig}
\usepackage{url}
\usepackage{tikz}
\usepackage{subcaption}
\usepackage{booktabs}
\usepackage[hidelinks]{hyperref}

\usepackage{array}

\usepackage{graphicx,bm,xcolor}
\usetikzlibrary{arrows}

\newtheorem{assumption}{Assumption}
\newtheorem{notation}{Notation}

\begin{document}

\markboth{Rainer Schneckenleitner, Stefan Takacs}{Condition number bounds for IETI-DP methods that are explicit in $h$ and $p$}

%
\catchline{}{}{}{}{}
%

\title{Condition number bounds for IETI-DP methods \\ that are explicit in $h$ and $p$}

\author{Rainer Schneckenleitner}

\address{Institute of Computational Mathematics,\\
Johannes Kepler University Linz,\\
Altenberger Str.~69, 4040 Linz, Austria\\
schneckenleitner@numa.uni-linz.ac.at}

\author{Stefan Takacs}

\address{Johann Radon Institute Institute for Computational and Applied Mathematics (RICAM),\\
Austrian Academy of Sciences,\\
Altenberger Str.~69, 4040 Linz, Austria\\
stefan.takacs@ricam.oeaw.ac.at}

\maketitle

\selectlanguage{english}
\begin{abstract}
	We study the convergence behavior of Dual-Primal Isogeometric Tearing and
	Interconnecting (IETI-DP) methods for solving large-scale algebraic systems arising
	from multi-patch Isogeometric Analysis. We focus on the Poisson problem on two
	dimensional computational domains. We provide a convergence analysis that covers
	several choices of the primal degrees of freedom: the vertex values, the edge
	averages, and the combination of both. We derive condition number bounds that show
	the expected behavior in the grid size $h$ and that are quasi-linear in the
	spline degree $p$.
\end{abstract}



%
%
\section{Introduction}
\label{sec:1}

Isogeometric Analysis (IgA), see Ref.~\refcite{HughesCottrellBazilevs:2005},
is a method for solving partial differential equations
(PDEs) in a way that integrates better with standard computer aided design (CAD)
software than classical finite element (FEM) simulation. Both the computational
domain and the solution of the PDE are represented as linear combination of
tensor-product B-splines or non-uniform rational B-splines (NURBS). Since only simple domains
can be represented by just one such spline function, usually the computational domain
is decomposed into subdomains, in IgA usually called patches, where each patch is parameterized
with its own spline function.

For IgA, like
for any other discretization method for PDEs, fast iterative solvers are of interest.
Domain decomposition solvers are a natural choice for multi-patch IgA.
Nowadays, Finite Element Tearing and Interconnecting (FETI) methods have 
become the most popular non-overlapping domain decomposition methods. 
After the introduction of the FETI technology by C.~Farhat and F.-X.~Roux in 1991\cite{FarhatRoux:1991a}, many FETI versions have been developed for different applications, see, e.g., Refs.~\refcite{ToselliWidlund:2005a,Pechstein:2013a,KorneevLanger:2017a} for a comprehensive
presentation of domain decomposition methods and, in particular, of FETI methods.
The most advanced and most widely used versions are certainly the 
Dual-Primal FETI methods (FETI-DP), which have been introduced in Ref.~\refcite{FarhatLesoinneLeTallecPiersonRixen:2001a}.
The Balancing Domain Decomposition by Constraints (BDDC) methods 
introduced in Ref.~\refcite{Dohrmann:2003} can be seen as the primal counterpart of the
FETI-DP methods. In Ref.~\refcite{MandelDohrmannTezaur:2005a}, it has been proven
that the spectra of the preconditioned systems for FETI-DP and BDDC
are essentially the same. Moreover, that paper gives an abstract framework for
proving condition number bounds which we use in this paper.

The Ref.~\refcite{KleissPechsteinJuttlerTomar:2012} extends FETI-DP to isogeometric
discretizations. The proposed method is sometimes called Dual-Primal Isogeometric Tearing
and Interconnecting (IETI-DP) method. Convergence analysis for IETI-DP methods\cite{HoferLanger:2017c,HoferLanger:2019a} and BDDC methods for IgA\cite{Veiga}
has been provided for numerous contexts. The given condition number bounds
are explicit in grid and patch size, sometimes also in other
parameters, like diffusion parameters. Convergence theory that also covers
the dependence on the spline degree, is not available in the context of IgA.

Condition number estimates for FETI methods applied to spectral element discretizations
have been provided in Refs.~\refcite{KlawonnPavarino:2008,Pavarino:2006}, which show that
the condition number only grows poly-logarithmic in the polynomial degree. Comparable bounds
are also known for Schwarz type\cite{GuoCao:1997,SchoeberlMelenkPechsteinZaglmayr:2007}
and substructuring\cite{Ainsworth:1996,PavarinoWidlund:1996} methods for 
$hp$ finite element discretizations. For an overview and more references, see Ref.~\refcite{KorneevLanger:2015}.
However, for $hp$-FETI-DP algorithms, such bounds are not known to the authors.
Often the analysis for spectral methods is carried out by showing that the stiffness matrix
of interest is spectrally equivalent to a stiffness matrix for a suitably constructed low-order
problem. While such an analysis yields appealing upper bounds for spectral methods, a straight-forward extension
of the analysis to IgA does not yield reasonable statements concerning the dependence on
the spline degree. An alternative is to work directly on the function spaces of interest.
One ingredient for any FETI analysis are energy bounds for the discrete harmonic extension,
which can be done by
explicitly constructing bounded extension operators, cf. Ref.~\refcite{Nepomnyaschikh:1995}
for the case of $h$-FEM, Refs.~\refcite{Beuchler:2005,BeuchlerSchoberl:2005} for spectral elements,
and the paper on hand for spline spaces.

One of the strengths of IgA is $k$-refinement which allows the construction of discretizations
that show the approximation
power of a high-order method for the costs (in terms of the degrees of freedom) of
a low order method. Certainly, an efficient realization also requires a fast solver
that is (almost) robust in the spline degree $p$. In the last years, $p$-robust approximation error estimates
for single-patch domains\cite{TakacsTakacs:2015,SandeManniSpeelers:2019}
and multi-patch domains\cite{Takacs:2018}, as well as $p$-robust linear
solvers, like the fast diagonalization method\cite{SangalliTani:2016}
or multigrid methods for single-patch domains\cite{DonatelliGaroniManniSerraCapizzanoSpeleers:2017,HofreitherTakacs:2017,PedelaRiva:2018} and multi-patch domains\cite{Takacs:2018} have
been proposed.

In the present paper,
we provide a convergence analysis for a standard IETI-DP solver.
The convergence analysis covers three cases for the choice of the
primal degrees of freedom: Vertex values (Algorithm~A), edge averages (Algorithm~B)
and a combination of both (Algorithm~C). For all cases,  we
show that the condition number of the preconditioned system is bounded by
\[
	C \, p \, \left(1+\log p+\max_k \log \frac{H^{(k)}}{h^{(k)}}\right)^2,
\]
where $p$ is the spline degree, $H^{(k)}$ is the diameter of the $k$th patch,
$h^{(k)}$ is the grid size on the $k$th patch, and
the constant $C$ only depends on the geometry function, the quasi-uniformity of the
grids on the individual patches and the maximum number of patches sharing any vertex.
This means that the constant $C$ is independent
of the number of patches, their diameters, the grid sizes, the spline degree and
the smoothness of the splines, i.e., all choices of the smoothness between $C^0$
and $C^{p-1}$ are covered. 

The remainder of this paper is organized as follows.
In Section~\ref{sec:2}, we introduce the model problem and discuss its
isogeometric discretization.
The proposed IETI-DP solver is presented in Section~\ref{sec:3}.
In Section~\ref{sec:4}, we give the convergence theory.
Then, we proceed with numerical experiments that illustrate our theoretical findings 
in Section~\ref{sec:5}.
In Section~\ref{sec:6}, we conclude with some final remarks.

\section{The model problem}
\label{sec:2}
%
%
%

We consider a standard Poisson problem with homogeneous Dirichlet boundary
conditions as model problem.
Let $\Omega \subset \mathbb{R}^2$ be an open, bounded and simply connected domain with
Lipschitz boundary. For a given right-hand side $f \in L_2(\Omega)$, we want
to find a function $u \in H^1_0(\Omega)$
such that 
\begin{align}
	\label{continousVarProb}
	\underbrace{
		\int_{\Omega}^{} \nabla u(x) \cdot \nabla v(x) \; \mathrm dx
	}_{\displaystyle a(u,v):= }
	=
	\underbrace{
		\int_{\Omega}^{} f(x)\,v(x) \; \mathrm dx
	}_{\displaystyle \langle f,v\rangle:= }
	\quad \mbox{for all}\quad v \in H^1_0(\Omega).
\end{align}
Here and in what follows, we denote by $L_2(\Omega)$ and $H^s(\Omega)$, $s\in\mathbb R$, 
the usual Lebesgue and Sobolev spaces, respectively. $H^1_0(\Omega)\subset
H^1(\Omega)$ is the subspace of functions vanishing on~$\partial\Omega$,
the boundary of $\Omega$.
These spaces are equipped with the standard scalar products $(\cdot, \cdot)_{L_2(\Omega)}$
and $(\cdot, \cdot)_{H^1(\Omega)}:=(\nabla\cdot, \nabla\cdot)_{L_2(\Omega)}$, seminorms $|\cdot|_{H^s(\Omega)}$ and norms $\|\cdot\|_{L_2(\Omega)}$ and $\|\cdot\|_{H^s(\Omega)}$.

We assume that the physical domain $\Omega$ is composed of $K$ non-overlapping
patches $\Omega^{(k)}$, i.e., we have
\begin{align*}
	\overline{\Omega} = \bigcup_{k=1}^K \overline{\Omega^{(k)}} \quad \text{and}\quad
	\Omega^{(k)} \cap \Omega^{(\ell)} = \emptyset \quad\text{for all}\quad k \neq \ell,
\end{align*}
where $\overline{T}$ denotes the closure of the set $T$. Each patch
$\Omega^{(k)}$ is parameterized by a geometry mapping 
\begin{align}
	G_k:\widehat{\Omega}:=(0,1)^2 \rightarrow \Omega^{(k)}:=G_k(\widehat{\Omega}) \subset \mathbb{R}^2, 
\end{align}
which can be continuously extended to the closure of the parameter domain
$\widehat{\Omega}$. In IgA, the geometry mapping is typically represented using
B-splines or NURBS. For the analysis, we only require the following assumption.

\begin{assumption}
	\label{ass:nabla}
	There are patch sizes $H^{(k)}>0$ for $k=1,\ldots,K$ and a
	constant $C_1>0$ such that
	\begin{align*}
		\| \nabla G_k \|_{L_\infty(\widehat{\Omega}^{(k)})} \le C_1\, H^{(k)}
		\quad\text{and}\quad
		\| (\nabla G_k)^{-1} \|_{L_\infty(\widehat{\Omega}^{(k)})} \le C_1\, \frac{1}{H^{(k)}}
	\end{align*}
	holds for all $k=1,\ldots,K$.
\end{assumption}

The following assumption guarantees that the patches form an admissible decomposition, i.e.,
that there are no T-junctions.
\begin{assumption}
	\label{ass:conforming}
	For any two patch indices $k\not=\ell$, the set
	$\overline{\Omega^{(k)}} \cap \overline{\Omega^{(\ell)}}$
	is either a common edge (including the neighboring vertices),
	a common vertex, or empty.
\end{assumption}
If two patches $\Omega^{(k)}$ and $\Omega^{(\ell)}$ share a common edge,
we denote that edge by $\Gamma^{(k,\ell)}=\Gamma^{(\ell,k)}$, and its pre-images by
$\widehat{\Gamma}^{(k,\ell)}:=G_k^{-1}(\Gamma^{(k,\ell)})$ and
$\widehat{\Gamma}^{(\ell,k)}:=G_{\ell}^{-1}(\Gamma^{(k,\ell)})$. Moreover, we define
\[
\mathcal N_\Gamma(k):=
\{ \ell \;:\; \Omega^{(k)}\mbox{ and }\Omega^{(\ell)}
\mbox{ share an edge}\}.
\]
Analogously, if two patches $\Omega^{(k)}$ and $\Omega^{(\ell)}$ share
only a vertex,
we denote that vertex by $\textbf x^{(k,\ell)}=\textbf x^{(\ell,k)}$,
and its corresponding pre-images by
$\widehat{\textbf x}^{(k,\ell)}:=G_k^{-1}(\textbf x^{(k,\ell)})$ and
$\widehat{\textbf x}^{(\ell,k)}:=G_{\ell}^{-1}(\textbf x^{(k,\ell)})$. Moreover, we
define $\mathcal{P}(\textbf x):= \{k\;:\; \textbf x\in \overline{\Omega^{(k)}} \}$ and
\[
\mathcal N_{\textbf x}(k):=
\{ \ell \;:\; \Omega^{(k)}\mbox{ and }\Omega^{(\ell)}
\mbox{ share a vertex or an edge}\}
.
\]
For the analysis, we need that the number of neighbors is bounded.
\begin{assumption}
	\label{ass:neighbors}
	There is a constant $C_2>0$ such that
	$
	| \mathcal{N}_{\textbf x}(k)| \le C_2
	$ 
	for all patches $k=1,\ldots,K$.
\end{assumption}

The provided assumptions guarantee that the pre-images in the parameter domain $\widehat \Gamma_D^{(k)}
:= G_k^{-1}(\partial\Omega\cap\partial\Omega^{(k)})$ of the (Dirichlet) boundary $\Gamma_D=\partial
\Omega$ consists of whole edges.

Now, we introduce the isogeometric function spaces. Let $p \in \mathbb{N}:={1,2,3,\dots}$ be a given spline degree. For simplicity, we assume that the spline
degree is uniformly throughout the overall domain. We use B-splines as discretization
space. To keep the paper self-contained, we introduce the basic spline notation.
The splines are defined based on a $p$-open knot vector:
\[
\Xi=(\xi_1,\ldots,\xi_{n+p+1})
=(\underbrace{\zeta_1,\ldots,\zeta_1}_{\displaystyle m_1 },
\underbrace{\zeta_2,\ldots,\zeta_2}_{\displaystyle m_2},
\ldots,
\underbrace{\zeta_{N_Z},\ldots,\zeta_{N_Z}}_{\displaystyle m_{N_Z}}),
\]
where the multiplicities satisfy $m_1=m_{N_Z}=p+1$,
and $m_i\in\{1,\ldots,p\}$ for $i=~2,\ldots,N_Z-1$ and the breakpoints satisfy
$\zeta_1<\zeta_2<\cdots<\zeta_{N_Z}$. We call
\[
Z=(\zeta_1,\ldots,\zeta_{N_Z})
\quad\mbox{and}\quad
M=(m_1,\ldots,m_{N_Z})
\]
the vector of breakpoints associated to $\Xi$ and
the vector of multiplicities associated to $\Xi$, respectively.
We denote the standard B-spline basis as obtained by the Cox-de~Boor
formula, cf. (2.1) and (2.2) in Ref.~\refcite{Cottrell:Hughes:Bazilevs},
by $(B[p,\Xi,i])_{i=1}^n$. The corresponding
spline space is given as linear span of these basis functions, i.e.,
\[
S[p,\Xi]:=\text{span}\{B[p,\Xi,1],\ldots, B[p,\Xi,n] \}.
\]

To obtain the isogeometric function space, we choose for each patch two $p$-open
knot vectors $\Xi^{(k,1)}$ and $\Xi^{(k,2)}$ over $(0,1)$. On the parameter domain
$\widehat \Omega$, we define the tensor-product spline space by
$\widehat{V}^{(k)}$ and its transformation to the physical domain by
$V^{(k)}$:
\begin{equation}\label{eq:vkdef}
\widehat{V}^{(k)}
:= \{ v\in S[p,\Xi^{(k,1)}] \otimes S[p,\Xi^{(k,2)}]
\;:\;
v|_{\widehat \Gamma_D^{(k)}} = 0 \}	
\;\mbox{and}\;
V^{(k)} := \widehat{V}^{(k)} \circ G_k^{-1},
\end{equation}
where $v|_T$ denotes the restriction of $v$ to $T$ (trace operator).
We introduce a basis for the space $\widehat V^{(k)}$ by choosing
the basis functions of the standard B-spline basis
that vanish on the Dirichlet boundary~$\widehat \Gamma_D^{(k)}$. We order the total
number of $N^{(k)}=N_\mathrm{I}^{(k)}+N_\Gamma^{(k)}$ basis functions such that the first
$N_\mathrm{I}^{(k)}$ are supported only in the interior of the patch
and the following $N_\Gamma^{(k)}$ basis function contribute to the boundary
of the patch:
\begin{equation}\label{eq:basis:def}
\begin{aligned}
&\widehat \Phi^{(k)} := ( \widehat\phi_i^{(k)} )_{i=1}^{N^{(k)}},\\
&\{ \widehat \phi_i^{(k)} \}
= \{ \widehat \phi\,:\, \exists j_1,j_2\;:\;\widehat\phi(x,y)=B[p,\Xi^{(k,1)},j_1](x)\, B[p,\Xi^{(k,2)},j_2](y)\wedge
\widehat \phi|_{\widehat \Gamma_D^{(k)}}  = 0 \},
\\
&	
\widehat\phi^{(k)}_i|_{\partial \widehat \Omega}  = 0 \Leftrightarrow i \in\{1,\ldots, N_\mathrm{I}^{(k)}\},
\;
\mbox{and}
\;
\widehat\phi^{(k)}_i|_{\partial \widehat \Omega} \not = 0 \Leftrightarrow i \in N_\mathrm{I}^{(k)}+\{1,\ldots, N_\Gamma^{(k)}\}.
\end{aligned}
\end{equation}

Following the pull-back principle used for defining the function space on the physical
domain, we define the basis for $V^{(k)}$ by
\[
\Phi^{(k)} := ( \phi_i^{(k)} )_{i=1}^{N^{(k)}}
\quad\mbox{and}\quad
\phi_{i}^{(k)} := \widehat \phi_{i}^{(k)} \circ G_k^{-1}.
\]
We assume that the underlying grids on
each of the patches are quasi-uniform.
\begin{assumption}
	\label{ass:quasiuniform}
	There are grid sizes $\widehat{h}^{(k)}>0$ for $k=1,\ldots,K$
	and a constant $C_3>0$ such that
	\[
	C_3 \, \widehat{h}^{(k)} \le
	\zeta_{i+1}^{(k,\delta)} - \zeta_i^{(k,\delta)} \le \widehat{h}^{(k)}
	\]
	holds for all $i=1,\ldots,N_Z^{(k,\delta)}-1$ and all $\delta=1,2$.
\end{assumption}

The grid size on the physical domain is defined via
$h^{(k)}:=\widehat{h}^{(k)} H^{(k)}$.

To be able to set up a $H^1$-conforming discretization on the whole domain $\Omega$,
we assume that the function spaces are fully matching.
\begin{assumption}
	\label{ass:fullyMatching}
	For any two patches $\Omega^{(k)}$ and $\Omega^{(\ell)}$ sharing a common
	edge~$\Gamma^{(k,\ell)}$, the following statement holds true.
	For any basis function $\phi_i^{(k)}$ that does not vanish on~$\Gamma^{(k,\ell)}$,
	there is exactly one basis function $\phi_j^{(\ell)}$ such that they agree
	on~$\Gamma^{(k,\ell)}$, i.e.,
	\begin{equation}
	\label{fullyMatching}
	\phi_i^{(k)} |_{\Gamma^{(k,\ell)}} \not = 0
	\quad\Rightarrow\quad
	\exists j\;:\;
	\phi_i^{(k)} |_{\Gamma^{(k,\ell)}} 
	=
	\phi_j^{(\ell)} |_{\Gamma^{(k,\ell)}}.
	\end{equation}	
\end{assumption}
This assumption is satisfied if the the spline degrees $p$, the knot vectors,
and the geometry mappings agree on all interfaces.
Based on this assumption, we define the overall function space as
\begin{equation}\label{eq:vdef}
V := \lbrace u \in H^1_0(\Omega): u|_{\Omega^{(k)}} \in V^{(k)} \text{ for } k = 1, \dots, K \rbrace.
\end{equation}
Using these function spaces, we obtain the Galerkin discretization of the variational
problem~\eqref{continousVarProb}, which reads as follows. Find $u \in V$ such that 
\begin{align}
	\label{discreteVarProb}
	a(u,v)=\langle f,v\rangle
	\quad \mbox{for all}\quad v \in V,
\end{align}
where $a(\cdot,\cdot)$ and $\langle f,\cdot\rangle$ are as defined in~\eqref{continousVarProb}.
By choosing a basis for $V$, the variational problem~\eqref{discreteVarProb} can
be written in a matrix-vector form. For the construction of a IETI-DP method, we omit
this step and directly work with the discrete variational
problem~\eqref{discreteVarProb}.

\section{The IETI-DP solver}
\label{sec:3}
%
%
%

In this section, we derive a IETI-DP solver for the discretized variational
problem~\eqref{discreteVarProb}. First, we observe that the bilinear form $a(\cdot,\cdot)$
and the linear form $\langle f,\cdot \rangle$ are the sum of contributions of each of the patches,
i.e.,
\[
a(u,v) = \sum_{k=1}^K a^{(k)}(u,v),
\quad\mbox{where} \quad
a^{(k)}(u,v) := \int_{\Omega^{(k)}}^{} \nabla u(x) \cdot \nabla v(x) \; \mathrm dx 
\]
and
\[
\langle f,v\rangle = \sum_{k=1}^K \langle f^{(k)},v\rangle,
\quad\mbox{where} \quad
\langle f^{(k)},v\rangle := \int_{\Omega^{(k)}}^{} f(x) \, v(x) \; \mathrm dx .
\]
By discretizing $a^{(k)}(\cdot,\cdot)$ and $\langle f^{(k)},\cdot \rangle$
using the basis $\Phi^{(k)}$, we obtain the matrix-vector system
\begin{equation}\label{linsys:local}
A^{(k)}\, \underline{u}^{(k)} = \underline{f}^{(k)},
\end{equation}
where $A^{(k)} = [a^{(k)}(\phi_j^{(k)},\phi_i^{(k)})]_{i,j=1}^{N^{(k)}}$
is a stiffness matrix and
$\underline{f}^{(k)} = [ \langle f^{(k)},\phi_i^{(k)}\rangle]_{i=1}^{N^{(k)}}$
is a load vector.

Since we are interested in the solution of the original problem~\eqref{discreteVarProb}, we need to enforce continuity. Doing this
is an operation on the interfaces only. Note that the basis $\Phi^{(k)}$
has been defined such that the first~$N_\mathrm{I}^{(k)}$ basis functions are supported
only in the interior of the patch and the following~$N_\Gamma^{(k)}$ basis functions do
contribute to the boundary of the patch, see~\eqref{eq:basis:def}.
Following this decomposition, we obtain
\[
A^{(k)}
=
\left(
\begin{array}{cccc}
A_{\mathrm I \mathrm I}^{(k)} & A_{\mathrm I\Gamma}^{(k)}\\
A_{\Gamma \mathrm I}^{(k)} & A_{\Gamma\Gamma}^{(k)}\\
\end{array}
\right)
\quad
\mbox{and}
\quad
\underline{f}^{(k)}
=
\left(
\begin{array}{c}
\underline{f}_{\mathrm I}^{(k)}\\\underline{f}_\Gamma^{(k)}
\end{array}
\right).
\]
Using this decomposition, we rewrite the linear systems~\eqref{linsys:local} as
Schur complement:
\begin{equation}\label{eq:s:sys}
S^{(k)} \underline w^{(k)} = \underline{g}^{(k)},
\end{equation}
where
\begin{equation}\label{eq:s:def}
S^{(k)} := A_{\Gamma\Gamma}^{(k)} - A_{\Gamma \mathrm I}^{(k)}(A_{\mathrm I\mathrm I}^{(k)})^{-1} A_{\mathrm I\Gamma}^{(k)}
\quad\mbox{and}\quad
\underline{g}^{(k)} := \underline{f}_\Gamma^{(k)}
- A_{\Gamma \mathrm I}^{(k)} (A_{\mathrm I\mathrm I}^{(k)})^{-1}  \underline{f}_{\mathrm I}^{(k)}.
\end{equation}
The overall solution is then obtain by
\begin{equation}\label{eq:inner:def}
\underline{u}^{(k)}
=
\left(
\begin{array}{c}
\underline{u}_{\mathrm I}^{(k)}\\\underline{u}_\Gamma^{(k)}
\end{array}
\right)
=
\left(
\begin{array}{c}
(A_{\mathrm I\mathrm I}^{(k)})^{-1}
A_{\Gamma \mathrm I}^{(k)}
\underline{w}^{(k)}\\\underline{w}^{(k)}
\end{array}
\right).
\end{equation}
By combining the linear systems~\eqref{eq:s:sys} for all patches, we obtain
the linear system
\begin{equation}\label{eq:schursys}
S \underline w = \underline g,
\end{equation}
where
\[
S := \left(
\begin{array}{ccc}
S^{(1)} \\ & \ddots \\ && S^{(K)}
\end{array}	
\right),
\quad 
\underline w := \left(
\begin{array}{c}
\underline w^{(1)} \\  \vdots \\ \underline w^{(K)}
\end{array}	
\right)
\quad \mbox{and}\quad
\underline g := \left(
\begin{array}{c}
\underline g^{(1)} \\  \vdots \\ \underline g^{(K)}
\end{array}	
\right).
\]

As a next step, we give an interpretation of the linear system in a variational sense.
Let 
\[
W^{(k)} := \{ v|_{\partial \Omega^{(k)}} \,:\, v \in V^{(k)}\}
\quad\mbox{and}\quad
W := \prod_{k=1}^K W^{(k)}
\]
be the function spaces on the skeleton.
Let $\mathcal{H}_h^{(k)}:W^{(k)}\rightarrow V^{(k)}$ be the discrete harmonic extension,
i.e., such that
\[
(\mathcal{H}_h^{(k)} w^{(k)})|_{\partial \Omega^{(k)}} = w^{(k)},
\quad\mbox{and}\quad
a(\mathcal{H}_h^{(k)} w^{(k)},v^{(k)})=0
\quad\mbox{for all}\quad
v^{(k)}\in V_0^{(k)},
\]
where $V_0^{(k)}:=\{ v\in V^{(k)}\,:\, v|_{\partial\Omega^{(k)}} = 0 \}$.
The linear system~\eqref{eq:schursys} can be rewritten as follows. Find
$w =(w^{(1)},\ldots,w^{(K)}) \in W$ such that 
\[
\underbrace{\sum_{k=1}^K a^{(k)} ( \mathcal{H}_h^{(k)} w^{(k)}, \mathcal{H}_h^{(k)} q^{(k)} )}
_{\displaystyle s(w,q):=}
=
\underbrace{
	\sum_{k=1}^K \langle f, \mathcal{H}_h^{(k)} q^{(k)} \rangle}
_{\displaystyle \langle g,q\rangle:=}
\quad\mbox{for all}\quad
q = (q^{(1)},\ldots,q^{(K)}) \in W.
\]
The next step is to enforce continuity between the patches. So, for any
two patches $\Omega^{(k)}$ and $\Omega^{(\ell)}$ that share an edge
$\Gamma^{(k,\ell)}$ and all functions $\phi_i^{(k)}$ and $\phi_j^{(\ell)}$
such that
\[
\phi_i^{(k)}|_{\Gamma^{(k,\ell)}} = \phi_j^{(\ell)}|_{\Gamma^{(k,\ell)}}\not=0
,
\]
we introduce a constraint of the form
\begin{equation}\label{eq:def:b}
w^{(k)}_{i^*} - w^{(\ell)}_{j^*} = 0,
\end{equation}
where $i^* = i-N_\mathrm{I}^{(k)}$ and $j^* = j-N_\mathrm{I}^{(\ell)}$ are the corresponding indices.
If we choose the vertex values as primal degrees of freedom (see also Algorithms~A
and C below), we do not introduce such constraints for the functions associated
to the vertices, see Figure~\ref{fig:ommiting}. If only the edge averages are chosen
as primal degrees of freedom (see also Algorithm~B below), we additionally introduce
constraints of the form \eqref{eq:def:b} for the corresponding degrees of freedom
for the patches that share only a vertex. This means that we represent the
vertex values in a fully redundant way, see Figure~\ref{fig:fully}. 

Say, the number of constraints of the form~\eqref{eq:def:b} is $N_\Gamma$. Then,
we define a matrix $B \in \mathbb R^{N_\Gamma\times N}$ such
that each of the constraints~\eqref{eq:def:b} constitutes one row of the
linear system
\[
B \, \underline w=0.
\]
This means that each row of the matrix $B$ has exactly two non-vanishing
entries: one with value $1$ and one with value $-1$. We decompose the matrix
$B$ into a collection of patch-local matrices $B^{(1)},\ldots,B^{(K)}$
such that $B = (B^{(1)} \cdots B^{(K)})$.
\begin{figure}
	\begin{center}
		\begin{minipage}{0.35\linewidth}
			\begin{tikzpicture}
			\fill[gray!20] (-0.2,0) -- (1.5,0) -- (1.5,1.7) -- (-0.2,1.7);
			\fill[gray!20] (-0.2,-1) -- (1.5,-1) -- (1.5,-2.7) -- (-0.2,-2.7);
			\fill[gray!20] (4.2,0) -- (2.5,0) -- (2.5,1.7) -- (4.2,1.7);
			\fill[gray!20] (4.2,-1) -- (2.5,-1) -- (2.5,-2.7) -- (4.2,-2.7);
			
			\draw (-0.2,0) -- (1.5,0) -- (1.5,1.7) node at (0.7,0.8) {$\Omega_1$}; 
			\draw (-0.2,-1) -- (1.5,-1) -- (1.5,-2.7) node at (0.7,-1.85) {$\Omega_2$}; 
			\draw (4.2,0) -- (2.5,0) -- (2.5,1.7) node at (3.4,0.8) {$\Omega_3$}; 
			\draw (4.2,-1) -- (2.5,-1) -- (2.5,-2.7) node at (3.4,-1.85) {$\Omega_4$}; 
			
			\draw (1.5,0) node[circle, fill, inner sep = 2pt] (A1) {};
			\draw (1.5,0.75) node[circle, fill, inner sep = 2pt] (A2) {};
			\draw (1.5,1.5) node[circle, fill, inner sep = 2pt] (A3) {};
			\draw (0.75,0) node[circle, fill, inner sep = 2pt] (A4) {};
			\draw (0,0) node[circle, fill, inner sep = 2pt] (A5) {};
			
			\draw (2.5,0) node[circle, fill, inner sep = 2pt] (B1) {};
			\draw (2.5,0.75) node[circle, fill, inner sep = 2pt] (B2) {};
			\draw (2.5,1.5) node[circle, fill, inner sep = 2pt] (B3) {};
			\draw (3.25,0) node[circle, fill, inner sep = 2pt] (B4) {};
			\draw (4,0) node[circle, fill, inner sep = 2pt] (B5) {};
			
			\draw (1.5,-1) node[circle, fill, inner sep = 2pt] (C1) {};
			\draw (1.5,-1.75) node[circle, fill, inner sep = 2pt] (C2) {};
			\draw (1.5,-2.5) node[circle, fill, inner sep = 2pt] (C3) {};
			\draw (0.75,-1) node[circle, fill, inner sep = 2pt] (C4) {};
			\draw (0,-1) node[circle, fill, inner sep = 2pt] (C5) {};
			
			\draw (2.5,-1) node[circle, fill, inner sep = 2pt] (D1) {};
			\draw (2.5,-1.75) node[circle, fill, inner sep = 2pt] (D2) {};
			\draw (2.5,-2.5) node[circle, fill, inner sep = 2pt] (D3) {};
			\draw (3.25,-1) node[circle, fill, inner sep = 2pt] (D4) {};
			\draw (4,-1) node[circle, fill, inner sep = 2pt] (D5) {};
			
			\draw[<->, line width = 1pt, latex-latex]
			(A2) edge (B2) (A3) edge (B3)
			(A4) edge (C4) (A5) edge (C5)
			(B4) edge (D4) (B5) edge (D5)
			(C2) edge (D2) (C3) edge (D3);
			\end{tikzpicture}
			\captionof{figure}{Omitting vertices\\ (Algorithms~A and C) \label{fig:ommiting}}
		\end{minipage}\hspace{4em}
		\begin{minipage}{0.35\linewidth}
			\begin{tikzpicture}
			\fill[gray!20] (-0.2,0) -- (1.5,0) -- (1.5,1.7) -- (-0.2,1.7);
			\fill[gray!20] (-0.2,-1) -- (1.5,-1) -- (1.5,-2.7) -- (-0.2,-2.7);
			\fill[gray!20] (4.2,0) -- (2.5,0) -- (2.5,1.7) -- (4.2,1.7);
			\fill[gray!20] (4.2,-1) -- (2.5,-1) -- (2.5,-2.7) -- (4.2,-2.7);
			
			\draw (-0.2,0) -- (1.5,0) -- (1.5,1.7) node at (0.7,0.8) {$\Omega_1$}; 
			\draw (-0.2,-1) -- (1.5,-1) -- (1.5,-2.7) node at (0.7,-1.85) {$\Omega_2$}; 
			\draw (4.2,0) -- (2.5,0) -- (2.5,1.7) node at (3.4,0.8) {$\Omega_3$}; 
			\draw (4.2,-1) -- (2.5,-1) -- (2.5,-2.7) node at (3.4,-1.85) {$\Omega_4$}; 
			
			\draw (1.5,0) node[circle, fill, inner sep = 2pt] (A1) {};
			\draw (1.5,0.75) node[circle, fill, inner sep = 2pt] (A2) {};
			\draw (1.5,1.5) node[circle, fill, inner sep = 2pt] (A3) {};
			\draw (0.75,0) node[circle, fill, inner sep = 2pt] (A4) {};
			\draw (0,0) node[circle, fill, inner sep = 2pt] (A5) {};
			
			\draw (2.5,0) node[circle, fill, inner sep = 2pt] (B1) {};
			\draw (2.5,0.75) node[circle, fill, inner sep = 2pt] (B2) {};
			\draw (2.5,1.5) node[circle, fill, inner sep = 2pt] (B3) {};
			\draw (3.25,0) node[circle, fill, inner sep = 2pt] (B4) {};
			\draw (4,0) node[circle, fill, inner sep = 2pt] (B5) {};
			
			\draw (1.5,-1) node[circle, fill, inner sep = 2pt] (C1) {};
			\draw (1.5,-1.75) node[circle, fill, inner sep = 2pt] (C2) {};
			\draw (1.5,-2.5) node[circle, fill, inner sep = 2pt] (C3) {};
			\draw (0.75,-1) node[circle, fill, inner sep = 2pt] (C4) {};
			\draw (0,-1) node[circle, fill, inner sep = 2pt] (C5) {};
			
			\draw (2.5,-1) node[circle, fill, inner sep = 2pt] (D1) {};
			\draw (2.5,-1.75) node[circle, fill, inner sep = 2pt] (D2) {};
			\draw (2.5,-2.5) node[circle, fill, inner sep = 2pt] (D3) {};
			\draw (3.25,-1) node[circle, fill, inner sep = 2pt] (D4) {};
			\draw (4,-1) node[circle, fill, inner sep = 2pt] (D5) {};
			
			\draw[<->, line width = 1pt, latex-latex]
			(A1) edge (D1) (A1) edge (B1) (A1) edge (C1) (B1) edge (C1) (C1) edge (D1) (B1) edge (D1)
			(A2) edge (B2) (A3) edge (B3)
			(A4) edge (C4) (A5) edge (C5)
			(B4) edge (D4) (B5) edge (D5)
			(C2) edge (D2) (C3) edge (D3);
			\end{tikzpicture}
			\captionof{figure}{Fully redundant\\ (Algorithm B)}
			\label{fig:fully}
		\end{minipage}
	\end{center}
\end{figure}

The original variational problem~\eqref{discreteVarProb} is equivalent to
the following problem. Find $(\underline{w},\underline{\lambda})$ such that
\[
\left(
\begin{array}{cccc}
S& B^\top \\
B			
\end{array}
\right)
\left(
\begin{array}{c}
\underline{w}  \\
\underline{\lambda}  \\
\end{array}
\right)
=
\left(
\begin{array}{c}
\underline{g}  \\
0  \\
\end{array}
\right).
\]
For patches $\Omega^{(k)}$ that do not contribute to the Dirichlet boundary of
the physical domain $\Omega$, the corresponding matrices $A^{(k)}$ and $S^{(k)}$
refer to Poisson problems with pure Neumann boundary conditions, which means that
these matrices are singular. So, in general, the matrix $S$ is singular as well.

To overcome this problem, primal degrees of freedom are introduced. Here, we
have several possibilities:
\begin{itemize}
	\item \emph{Algorithm A (Vertex values):} The space $\widetilde{W}$
	is the subspace of functions where the vertex values agree, i.e.,
	\[
	\widetilde{W} :=
	\left\{ w\in W
	\;:\; 
	\begin{array}{l}
	w^{(k)}(\textbf x) = w^{(\ell)}(\textbf x) \\
	\mbox{for all common vertices $\textbf x$ of all pairs of
		$\Omega^{(k)}$ and $\Omega^{(\ell)}$}
	\end{array}
	\right\}.
	\]	
	The subspace $\widetilde{W}_{\Delta} \subset \widetilde{W}$ satisfies these conditions homogeneously, i.e.,
	$\widetilde{W}_{\Delta}:= \prod_{k=1}^K \widetilde{W}_{\Delta}^{(k)}$ and
	\[
	\widetilde{W}_{\Delta}^{(k)} := \{ w\in W^{(k)} \;:\; w(\textbf x) = 0 \mbox{ for all vertices } \textbf x \mbox{ of all } \Omega^{(k)} \}.
	\] 
	\item \emph{Algorithm B (Edge averages):} The space $\widetilde{W}$
	is the subspace of functions where the averages of the function values over
	the edges agree, i.e.,
	\[
	\widetilde{W} :=
	\left\{ w\in W
	\;:\; \int_{\Gamma^{(k,\ell)}} w^{(k)}(x)\,\mathrm dx
	= \int_{\Gamma^{(k,\ell)}} w^{(\ell)}(x)\,\mathrm dx
	\mbox{ for all edges }\Gamma^{(k,\ell)} \right\}.
	\]	
	The B-spline space $\widetilde{W}_{\Delta}$ satisfies these conditions homogeneously, i.e.,
	$\widetilde{W}_{\Delta}:= \prod_{k=1}^K \widetilde{W}_{\Delta}^{(k)}$ and
	\[
	\widetilde{W}_{\Delta}^{(k)} := \left\{ w\in W^{(k)} \;:\; \int_{\Gamma^{(k,\ell)}} w^{(k)}(x)\,\mathrm dx = 0 \mbox{ for all edges } \Gamma^{(k,\ell)} \right\}.
	\] 
	\item \emph{Algorithm C (Vertex values and edge averages):} We combine the constraints from
	both cases. So, the spaces $\widetilde{W}$ and $\widetilde{W}_{\Delta}^{(k)}$ and
	$\widetilde{W}_{\Delta}$ are the intersections of the corresponding spaces obtained by
	Algorithms~A and B.
\end{itemize}

We introduce matrices $C^{(k)}$ representing the subspace $\widetilde W^{(k)}_\Delta$, i.e.,
we have
\[
C^{(k)} \underline w^{(k)} = 0
\quad \Leftrightarrow \quad
w^{(k)} \in \widetilde{W}_\Delta^{(k)}
\]
for all $w^{(k)} \in W^{(k)}$ with vector representation $\underline w^{(k)}$. The
matrix $C^{(k)}$ is chosen to have full rank, i.e., the number of rows coincides
with the number of primal degrees of freedom per patch. The matrix $C$ is a
block-diagonal matrix containing the blocks $C^{(1)},\ldots,C^{(K)}$.

For every choice of primal degrees of freedom, the space $\widetilde{W}_{\Pi}$ is the
$S$-orthogonal complement of $\widetilde{W}_{\Delta}$ in $\widetilde{W}$, i.e.,
\begin{equation}\label{eq:wtildedef}
\widetilde{W}_{\Pi} := \{ w\in \widetilde{W} \;:\; 
s( w, q ) = 0
\mbox{ for all } q \in \widetilde{W}_{\Delta} \}.
\end{equation}
Let $\psi^{(1)},\ldots,\psi^{(N_\Pi)}$ be a basis of $\widetilde{W}_{\Pi}$. For the computation, one
usually chooses a nodal basis, where the vertex values and/or the edge averages form
the nodal values. The matrix $\Psi$ represents the basis in terms of the basis
for the space $W$, i.e.,
\begin{equation}\label{eq:psidef}
\Psi = \left(
\begin{array}{c}
\underline\psi^{(1)}\\\vdots\\\underline\psi^{(N_\Pi)}
\end{array}
\right).
\end{equation}

Following Ref.~\refcite{MandelDohrmannTezaur:2005a}, the following linear system is an equivalent
rewriting of the original variational problem~\eqref{discreteVarProb}: Find
$(\underline w_\Delta, \underline \mu, \underline w_\Pi, \underline \lambda)$ such that
\[
\left(
\begin{array}{cccc}
S & C^\top &                  & B^\top           \\
C &        &                  &                  \\
&        & \Psi^\top S \Psi & \Psi^\top B^\top \\
B &        & B \Psi           &                  \\
\end{array}
\right)
\left(
\begin{array}{c}
\underline{w}_\Delta  \\
\underline{\mu}  \\
\underline{w}_\Pi  \\
\underline{\lambda}  \\
\end{array}
\right)
=
\left(
\begin{array}{c}
\underline{g}  \\
0  \\
\Psi^\top \underline{g}  \\
0  \\
\end{array}
\right),
\]
where the solution for the original problem is obtained by
$\underline{w} = \underline{w}_\Delta + \Psi \underline{w}_\Pi$.
By reordering and by building the Schur complement, we obtain the following equivalent formulation
\begin{equation}
\label{IETIProblem}
F \; \underline{\lambda} = \underline{d},
\end{equation}
where
\begin{align}
	\label{eq:IETI-matrix}
	F :=  \underbrace{\left(
		\begin{array}{ccc}
			B & 0 & B \Psi
		\end{array}
		\right)
		\left(
		\begin{array}{ccc}
			S & C^\top &  \\ C &   &   \\   &   & \Psi^\top S\Psi
		\end{array}
		\right)^{-1}
		\left(
		\begin{array}{c}
			I  \\ 0 \\ \Psi^\top 
		\end{array}
		\right)}_{\displaystyle F_0:=}
	B^\top
	\quad\mbox{and}\quad
	\underline{d} :=  F_0\, \underline{g}.
\end{align}

The idea of the IETI-DP method is to solve the linear system~\eqref{IETIProblem}
with a preconditioned conjugate gradient (PCG) solver. As preconditioner, we
use the scaled Dirichlet preconditioner $M_{\mathrm{sD}}$, which is given by
\[
M_{\mathrm{sD}} := B D^{-1} S  D^{-1} B^\top ,
\]
where $D \in \mathbb{R}^{N_\Gamma\times N_\Gamma}$ is a diagonal matrix
defined based on the principle of multiplicity scaling. This means that each
coefficient $d_{i,i}$ of $D$ is assigned the number of Lagrange multipliers 
that act on the corresponding basis function, but at least 1, i.e., 
\[
d_{i,i} := \max\Big\{ 1,
\sum_{j=1}^{N_\Gamma} b_{i,j}^2
\Big\},
\]
where $b_{i,j}$ are the coefficients of the matrix $B$.

For the realization of the proposed method, one has to perform the following
steps.
\begin{itemize}
	\item Compute the vectors $\underline g^{(k)}$ according to \eqref{eq:s:def}.
	\item Compute the matrix $\Psi$ according to~\eqref{eq:wtildedef} and~\eqref{eq:psidef}.
	\item Compute $S_\Pi := \Psi^{\top} S \Psi$
	\item Execute a preconditioned conjugate gradient (PCG) solver
	for computing $\underline \lambda$. 
	This requires the computation of the residual and the application of
	the preconditioner. For the computation of the residual
	$\underline{\widehat w}:=F \underline \lambda - \underline d$,
	the following steps are applied:
	\begin{itemize}
		\item Compute $\underline{\widehat q}
		= ((\underline{\widehat q}^{(1)})^\top \cdots (\underline{\widehat q}^{(K)})^\top)^\top
		:= B^\top \underline \lambda - \underline g$.
		\item Solve the linear system
		\begin{equation}\label{eq:algo1}
		\left(
		\begin{array}{cc}
		S^{(k)} & (C^{(k)})^\top \\ C^{(k)}
		\end{array}
		\right)
		\left(
		\begin{array}{c}
		\underline{\widehat w}_\Delta^{(k)} \\ \underline{\widehat \mu}^{(k)}
		\end{array}
		\right)
		=
		\left(
		\begin{array}{c}
		\underline{\widehat q}^{(k)} \\ 0
		\end{array}
		\right)
		\end{equation}
		for all $k=1,\ldots,K$. If the vertex values are chosen as primal degrees of
		freedom (Algorithms A and C), it is possible to get an equivalent formulation by
		eliminating the degrees of freedom corresponding to the vertex values and the
		corresponding Lagrange multipliers.
		\item Solve the linear system
		\begin{equation}\label{eq:algo2}
		S_\Pi \underline{\widehat w}_\Pi = \Psi^\top \underline{\widehat q},
		\end{equation}
		which is usually expected to be small.
		\item The residual is given by
		\begin{equation}\label{eq:algo3}
		\underline{\widehat w} := B \left(
		\begin{array}{ccc}
		\underline{\widehat w}_\Delta^{(1)} \\ \vdots \\ \underline{\widehat w}_\Delta^{(K)}
		\end{array}
		\right)
		+ B \Psi  \underline {\widehat w}_\Pi.
		\end{equation}
	\end{itemize}
	The computation of the preconditioned residual $\underline{\widehat p} := M_{\mathrm{sD}}\,
	\underline{\widehat w}$ only requires matrix-vector multiplications.
	\item To obtain the solution vector $\underline u$, first $\underline w$
	is computed analogously to~\eqref{eq:algo1}, \eqref{eq:algo2} and \eqref{eq:algo3}
	based on $q=((q^{(1)})^\top \cdots (q^{(K)})^\top)^\top
	:=B^\top \underline \lambda$. Then,
	the solution vector $\underline u$ is obtained by~\eqref{eq:inner:def}.
\end{itemize}

\section{Condition number estimate}
\label{sec:4}
In this section, we prove the following condition number estimate.

\begin{theorem}\label{thrm:fin}
	Provided that the IETI-DP solver is set up as outlined in the previous
	sections, the condition number of the preconditioned system satisfies
	\[
		\kappa(M_{\mathrm{sD}} F) \le C\, p \left(1+\log p+\max_{k=1,\ldots,K} \log\frac{H^{(k)}}{h^{(k)}}\right)^2,
	\]
	where $C$ only depends on the constants from the
	Assumptions~\ref{ass:nabla}, \ref{ass:neighbors}, and~\ref{ass:quasiuniform}.
\end{theorem}

Note that, for a fixed choice of the spline degree $p$, this estimate behaves
essentially the same as one would expect for standard low-order finite elements.

\begin{notation}
	We use the notation $a \lesssim b$ if there is a constant
	$c>0$ that only depends on the constants from the
	Assumptions~\ref{ass:nabla}, \ref{ass:neighbors}, and~\ref{ass:quasiuniform}
	such that $a\le cb$. Moreover, we write $a\eqsim b$ if $a\lesssim b\lesssim a$.
\end{notation}

Following the standard path,
we develop an analysis in the $H^{1/2}$-seminorm on the skeleton. For this, we need
to know that the $H^1$-seminorm of the discrete harmonic extension is bounded by
the $H^{1/2}$-seminorm on the boundary. Since such a result has not been worked
out for spline spaces, we give an estimate in Section~\ref{sec:4:1}. In Section~\ref{sec:4:2},
we give a $p$-robust embedding statement and in Section~\ref{sec:4:3}, we give
a lemma that allows to tear the $H^{1/2}$-seminorm apart. In
Section~\ref{sec:4:fin}, we use these results to give a condition number bound.

\subsection{Estimates for the discrete harmonic extension}\label{sec:4:1}

The analysis of this subsection follows Ref.~\refcite{Nepomnyaschikh:1995}, where
the same technique has been used to estimate the discrete harmonic
extension in the context of Finite Element methods. In this subsection, we only consider
one single patch at a time, i.e., $k$ is fixed. Let
\[
\mathcal V := S[p,\Xi^{(k,1)}] \otimes S[p,\Xi^{(k,2)}]
\quad\mbox{and}\quad
\mathcal W := \{ v|_{\partial \widehat\Omega}\,:\,
v\in \mathcal V\}
\]
be the corresponding function space and its restriction to $\partial \widehat\Omega$,
respectively. Note that $\widehat{V}^{(k)} = \{ v \in \mathcal V \;:\;
v|_{\widehat \Gamma_D^{(k)}}=0\}$.

Let $Z^{(k,\delta)}=(\zeta_1^{(k,\delta)},\ldots,\zeta_{N_Z^{(k,\delta)}}^{(k,\delta)})$
be the vectors of breakpoints associated to the knot 
vectors $\Xi^{(k,\delta)}$, for $\delta=1,2$, respectively. Let $\widehat{h}$ be
the corresponding grid size, using Assumption~\ref{ass:quasiuniform}, we have
\[
\widehat{h} \lesssim
\zeta^{(k,\delta)}_{i+1}-\zeta^{(k,\delta)}_{i} \le \widehat{h}
\]
for all $i=1,\ldots,N_Z^{(k,\delta)}-1$ and all $\delta=1,2$. As a next step, we introduce
a hierarchy of nested grids
\[
Z^{(k,\delta,0)} := (0,1), \quad Z^{(k,\delta,1)}, \quad \cdots,
\quad Z^{(k,\delta,L)} := Z^{(k,\delta)}
\]
with $Z^{(k,\delta,\ell)} = (\zeta_1^{(k,\delta,\ell)},\ldots,\zeta_{N_Z^{(k,\delta,\ell)}}^{(k,\delta,\ell)})$
such that $\zeta^{(k,\delta,\ell)}_{1}=0$, $\zeta^{(k,\delta,\ell)}_{N_Z^{(k,\delta,\ell)}}=1$
and
\begin{equation}\label{eq:gridsize}
\widehat{h}_\ell  \lesssim
\zeta^{(k,\delta,\ell)}_{i+1}-\zeta^{(k,\delta,\ell)}_{i} \le  \widehat{h}_\ell,
\quad\mbox{where}\quad
\widehat{h}_\ell:=4^{L-\ell}\, \widehat{h}
\end{equation}
for all $i=1,\ldots,N_Z^{(k,\delta,\ell)}-1$, all $\ell=1,\ldots,L$
and all $\delta=1,2$. The following Lemma guarantees the existence of such grids.
\begin{lemma}\label{lem:coarse}
	Let $Z := (\zeta_1,\ldots,\zeta_{N_Z})$ with $\zeta_1=0$ and $\zeta_{N_Z}=1$
	be a given vector of breakpoints with grid
	size $h:=\max_{i=1,\ldots,{N_Z}-1} \zeta_{i+1} - \zeta_{i}$.
	For all $\widetilde h \ge 3h$, there is a vector
	of breakpoints $\widetilde{Z}
	:= (\widetilde{\zeta}_0,\ldots,\widetilde{\zeta}_{N_{\widetilde Z}})$ such that
	\begin{equation}\label{eq:z}
	\min\{1,\tfrac13\widetilde h\}
	\le
	\widetilde{\zeta}_{i+1} - \widetilde{\zeta}_{i}
	\le
	\widetilde h
	\quad\mbox{for all}\quad
	i=1,\ldots,N_{\widetilde Z}-1.
	\end{equation}
\end{lemma}
\begin{proof}
	Let $\widetilde h \ge 3h$ be arbitrary but fixed.
	Let $\widehat\zeta_0:=0$ and
	\begin{equation}\label{eq:zetastardef}
	\widehat\zeta_i
	:=
	\max \{ \zeta_j \;:\; \zeta_j \le \widehat\zeta_{i-1} + \tfrac 23\widetilde h \}
	\quad \mbox{for}\quad i=1,2,\ldots
	\end{equation}
	Let $N_{\widehat Z}$ be the smallest index such that $\widehat\zeta_{N_{\widehat Z}}=1$. Define
	\begin{equation}\label{eq:zetatildedef}
	\widetilde Z :=
	(\widetilde{\zeta}_1,\ldots,\widetilde{\zeta}_{N_{\widetilde Z}})
	:= \begin{cases}
	(\widehat\zeta_1,\ldots,\widehat\zeta_{N_{\widehat Z}})
	& \mbox{if } \widehat\zeta_{N_{\widehat Z}} - \widehat\zeta_{N_{\widehat Z}-1} \ge \min\{1, \tfrac13\widetilde h\} \\
	(\widehat\zeta_1,\ldots,\widehat\zeta_{N_{\widehat Z}-2},\widehat\zeta_{N_{\widehat Z}})
	& \mbox{otherwise.} \\
	\end{cases}
	\end{equation}
	Here, we take the minimum of $1$ and $\tfrac12\widetilde h$ in the first case to make this definition
	to be correct also for the case $N_{\widehat Z}=1$.
	From~\eqref{eq:zetastardef}, we immediately obtain that
	$\widehat\zeta_{i+1} - \widehat\zeta_i \le \tfrac23\widetilde h$
	for all $i=1,\ldots,N_{\widehat Z}$. From this
	and the construction in~\eqref{eq:zetatildedef}, we obtain the upper
	bound in~\eqref{eq:z}.
	
	Observe that $\widehat\zeta_{i+1} - \widehat\zeta_i \ge \tfrac23\widetilde h-h \ge \tfrac13\widetilde h$
	for all $i=1,\ldots,N_{\widehat Z}-1$. This follows from the definition in~\eqref{eq:zetastardef}
	in combination with the fact that the grid size of the original grid is $h$.
	This immediately implies $\widetilde{\zeta}_{i+1} - \widetilde{\zeta}_i \ge \tfrac13\widetilde h$
	for all $i=1,\ldots,N_{\widehat Z}-1$. If $N_{\widetilde Z} = N_{\widehat Z}-1$,
	this finishes the proof.	
	If $N_{\widetilde Z} = N_{\widehat Z}$, we have by assumption that
	$\widetilde{\zeta}_{N_{\widetilde Z}} - \widetilde{\zeta}_{N_{\widetilde Z}-1}
	\ge \min\{1, \tfrac13 \widetilde h\}$,
	which finishes the proof also in this case.
\end{proof}

For each of this vectors of breakpoints, we introduce corresponding
$p$-open knot vectors without repeated inner knots:
\[
\Xi^{(k,\delta,\ell)} := (\underbrace{\zeta_1^{(k,\delta,\ell)},\ldots,\zeta_1^{(k,\delta,\ell)}}_{\displaystyle p+1\mbox{ times}},
\zeta_2^{(k,\delta,\ell)},
\ldots,
\zeta_{N_Z^{(k,\delta,\ell)}-1}^{(k,\delta,\ell)},
\underbrace{\zeta_{N_Z^{(k,\delta,\ell)}}^{(k,\delta,\ell)}, \ldots, \zeta_{N_Z^{(k,\delta,\ell)}}^{(k,\delta,\ell)}}_{\displaystyle p+1 \mbox{ times}}).
\]
Based on these grids, we introduce coarse-grid spline spaces of maximum smoothness by
\[
\mathcal W^{(\ell)} := \left\{
w \in (S[p,\Xi^{(k,1,\ell)}]\otimes S[p,\Xi^{(k,2,\ell)}])|_{\partial\widehat\Omega}
\,:\, w\circ\gamma \in C^{p-1}(-\infty,\infty)
\right\}
\]
where $\gamma:(-\infty,\infty) \rightarrow \partial\widehat\Omega$ is given by
\begin{equation}\label{eq:unroll}
\gamma(t) := \begin{cases}
(0,t)     & \mbox{ if } t\in [0,1) + 4\mathbb Z \\
(t-1,1)   & \mbox{ if } t\in [1,2) + 4\mathbb Z\\
(1,3-t)   & \mbox{ if } t\in [2,3) + 4\mathbb Z\\
(4-t,0)   & \mbox{ if } t\in [3,4) + 4\mathbb Z\\
\end{cases},
\end{equation}
where $\mathbb Z$ is the set of integers. Note that the condition
$w\circ\gamma \in C^{p-1}(-\infty,\infty)$ introduces special
smoothness requirements on the vertices of $\widehat \Omega$.
For convenience, we define
$\mathcal W^{(0)}:=\{0\}$ and $\mathcal W^{(L+1)} := \mathcal W$
and observe that the spaces $\mathcal W^{(\ell)}$ are nested, i.e.,
\[
\mathcal W^{(0)} \subset
\mathcal W^{(1)} \subset
\cdots \subset
\mathcal W^{(L)} \subset
\mathcal W^{(L+1)}.
\]

Since the spaces $\mathcal W^{(\ell)}$ for all $\ell=1,\ldots,L$
are periodic and of maximum smoothness, we have a robust inverse estimate.
The space $\mathcal W^{(L+1)}$ is not a spline space of maximum smoothness,
thus only a standard inverse estimate for piecewise polynomial functions can be used.
\begin{lemma}\label{lem:inverse:q}
	The following $H^1-L_2$ and $H^2-H^1$-inverse estimates hold:
	\begin{itemize}
		\item
		$|w|_{H^1(\partial\widehat\Omega)}^2
		\lesssim
		p^4\, \widehat h_{L+1}^{-2} \|w\|_{L_2(\partial\widehat\Omega)}^2$
		for $w\in \mathcal W^{(L+1)}$.
		\item
		$|w|_{H^1(\partial\widehat\Omega)}^2
		\lesssim
		\widehat h_{\ell}^{-2} \|w\|_{L_2(\partial\widehat\Omega)}^2$
		for $w\in \mathcal W^{(\ell)}$ with $\ell=1,\ldots,L$.
	\end{itemize}
\end{lemma}
\begin{proof}
	This follows immediately from definition~\eqref{eq:unroll}, the observation that
	$|w|_{H^1(\partial\widehat\Omega)} = |w\circ \gamma|_{H^1(0,4)}$
	and $\|w\|_{L_2(\partial\widehat\Omega)} = \|w\circ \gamma\|_{L_2(0,4)}$,
	the inverse inequalities given in Theorem~4.76, eq. (4.6.5) in
	Ref.~\refcite{Schwab} (for $\ell=L+1$)
	and Theorem~6.1 in Ref.~\refcite{TakacsTakacs:2015} (for $\ell=1,\ldots,L$) and~\eqref{eq:gridsize}.
\end{proof}

The following Lemma shows that there is also an $H^1-H^{1/2}$-inverse estimate.

\begin{lemma}\label{lem:inverse:half}
	The $H^1-H^{1/2}$-inverse estimate
	$|w|_{H^1(\partial\widehat\Omega)}^2
	\lesssim p^2\, \widehat h_{L+1}^{-1} |w|_{H^{1/2}(\partial\widehat\Omega)}^2$ 
	holds for all $w\in \mathcal{W}^{(L+1)}$.
\end{lemma}
\begin{proof}
	Using the reiteration theorem and the fact that the fractional order Sobolev spaces
	coincide with the corresponding interpolation spaces,
	cf. Theorems~7.21 and 7.31 in Ref.~\refcite{AdamsFournier:2003},
	we obtain that $H^1$ is the interpolation between the Sobolev spaces $H^{1/2}$ and $H^2$,
	i.e.,
	\[
	\|w\|_{H^1(\partial\widehat\Omega)} \eqsim \|w\|_{[H^{1/2}(\partial\widehat\Omega),H^2(\partial\widehat\Omega)]_{1/3}}.
	\]
	Thus, we obtain
	\[
	\|w\|_{H^1(\partial\widehat\Omega)}^3 \lesssim
	\|w\|_{H^{1/2}(\partial\widehat\Omega)}^2
	\|w\|_{H^2(\partial\widehat\Omega)}.
	\]
	Since the derivative of a spline is again a spline, we obtain analogously
	to the proof of Lemma~\ref{lem:inverse:q} that
	$\|w\|_{H^2(\partial\widehat\Omega)} \lesssim p^2 \widehat h_{L+1}^{-1} \|w\|_{H^1(\partial\widehat\Omega)}$.
	This yields the estimate
	$
	\|w\|_{H^1(\partial\widehat\Omega)}^2 \lesssim p^2 \widehat h_{L+1}^{-1}
	\|w\|_{H^{1/2}(\partial\widehat\Omega)}^2
	$. The Poincar\'e inequality finishes the proof.
\end{proof}

Let $Q_\ell$ be the $L_2$-orthogonal projector into $\mathcal W^{(\ell)}$
and let $\Pi_\ell$ be the $H^1$-orthogonal
projector into $\mathcal W^{(\ell)}$, minimizing the distance in the norms
$\|\cdot\|_{L_2(\partial\widehat\Omega)}^2$ and
$\|\cdot\|_{H^1(\partial\widehat\Omega)}^2$, respectively.
For these projectors, the following error estimates are satisfied.

\begin{lemma}\label{lem:approx:q}
	The estimate
	$\| (I-Q_\ell) w \|_{L_2(\partial\widehat\Omega)}^2 \lesssim \widehat h_\ell
	\|w\|_{H^{1/2}(\partial\widehat\Omega)}^2$ holds for all functions
	$w\in H^1(\partial\widehat\Omega)$ and all $\ell=0,\ldots,L+1$.
\end{lemma}
\begin{proof}
	For $\ell=0$, we have $Q_\ell=0$ and $\widehat h_\ell=1$. Thus, the statement
	is trivial. Now, assume that $\ell \ge 1$.
	Let $w\in H^1(\partial\widehat\Omega)$ be arbitrary but fixed and let
	$v:=w\circ \gamma$.
	Observe that
	\[
	\| (I-Q_\ell) w \|_{L_2(\partial\widehat\Omega)}^2
	= \inf_{w_\ell \in \mathcal W^{(\ell)}}
	\| w-w_\ell \|_{L_2(\partial\widehat\Omega)}^2
	= \inf_{v_\ell \in \mathcal V^{(\ell)}}
	\| v-v_\ell \|_{L_2(0,4)}^2,
	\]
	where $\gamma$ is as in~\eqref{eq:unroll}.
	The Theorem~4.1 in Ref.~\refcite{SandeManniSpeelers:2019} and~\eqref{eq:gridsize} yield
	\[
	\| (I-Q_\ell) w \|_{L_2(\partial\widehat\Omega)}^2
	\lesssim \widehat h_\ell^2 |v|_{H^1(0,4)}^2.
	\]
	For the choice $v_\ell:=0$, we have
	\[
	\| (I-Q_\ell) w \|_{L_2(\partial\widehat\Omega)}^2 \lesssim \|v\|_{L_2(0,4)}^2.
	\]
	Using Hilbert space interpolation theory, cf. Theorems~7.23 and 7.31 in Ref.~\refcite{AdamsFournier:2003}, we obtain
	\[
	\| (I-Q_\ell) w \|_{L_2(\partial\widehat\Omega)}^2 \lesssim \widehat h_\ell \|v\|_{H^{1/2}(0,4)}^2.
	\]
	Using $\|v\|_{L_2(0,4)}^2= \|w\|_{L_2(\partial\widehat\Omega)}^2$ and
	\begin{equation}\nonumber
	|v|_{H^{1/2}(0,4)}^2
	= \int_0^4\int_0^4 \frac{|v(s)-v(t)|^2}{|s-t|^2} \,\mathrm ds \,\mathrm dt
	\lesssim
	\int_0^4\int_0^4 \frac{|v(s)-v(t)|^2}{|\gamma(s)-\gamma(t)|^2} \,\mathrm ds \,\mathrm dt
	= |w|_{H^{1/2}(\partial\widehat\Omega)}^2
	,
	\end{equation}
	we immediately obtain the desired result.
\end{proof}

\begin{lemma}\label{lem:approx:pi}
	$\| (I-\Pi_\ell) w \|_{L_2(\partial\widehat\Omega)}^2 \lesssim \widehat h_\ell^2
	|w|_{H^1(\partial\widehat\Omega)}^2$
	holds for all $w\in H^1(\partial\widehat\Omega)$ and all $\ell=0,\ldots,L+1$.
\end{lemma}
\begin{proof}
	Analogously to the proof of Lemma~\ref{lem:approx:q}, one can derive from
	Theorem~4.1 in Ref.~\refcite{SandeManniSpeelers:2019} and~\eqref{eq:gridsize} that
	$
	| (I-\Pi_\ell) w |_{H^1(\partial\widehat\Omega)}^2 \lesssim \widehat h_\ell^2
	|w|_{H^{2}(\partial\widehat\Omega)}^2
	$
	holds for all $w\in H^1(\partial\widehat\Omega)$.
	The desired result follows using a standard Aubin-Nitsche duality argument.
\end{proof}

The following Lemma shows that
an $L_2$-orthogonal decomposition almost realizes the minimum of arbitrary
decompositions.
\begin{lemma}\label{lem:qsum:any}
	The estimate
	\[
	\sum_{\ell=1}^{L+1} \widehat h_\ell^{-2}
	\| (Q_\ell-Q_{\ell-1}) w \|_{L_2(\partial\widehat\Omega)}^2
	\eqsim
	\inf_{w_\ell\in \mathcal W^{(\ell)},\, w=\sum_{\ell=1}^{L+1} w_{\ell}}
	\sum_{\ell=1}^{L+1}
	\widehat h_{\ell}^{-2} \|w_{\ell}\|_{L_2(\partial\widehat\Omega)}^2
	\]
	holds for all $w\in H^1(\partial\widehat\Omega)$.
\end{lemma}
\begin{proof}
	Since $(Q_\ell-Q_{\ell-1}) w \in \mathcal W^{(\ell)}$, we immediately
	obtain
	\[
	\sum_{\ell=1}^{L+1} \widehat h_\ell^{-2}
	\| (Q_\ell-Q_{\ell-1}) w \|_{L_2(\partial\widehat\Omega)}^2
	\ge
	\inf_{w_\ell\in \mathcal W^{(\ell)},\, w=\sum_{\ell=1}^{L+1} w_{\ell}}
	\sum_{\ell=1}^{L+1}
	\widehat h_{\ell}^{-2} \|w_{\ell}\|_{L_2(\partial\widehat\Omega)}^2.
	\]
	For the proof of the other direction, let $w\in H^1(\partial\widehat\Omega)$
	be arbitrary but fixed. Consider any fixed representation of
	$w=\sum_{\ell=1}^{L+1} w_{\ell}$ with $w_\ell\in \mathcal W^{(\ell)}$.
	Observe that $w_{\ell}$ can be uniquely written as
	\[
	w_{\ell}
	= \sum_{n=1}^{\ell} w_{\ell}^{(n)}
	,
	\quad\mbox{where}\quad
	w_{\ell}^{(n)} = (Q_n-Q_{n-1}) w_\ell.
	\]
	From $w=\sum_{\ell=1}^{L+1} w_{\ell}$, we immediately obtain
	$
	(Q_n-Q_{n-1}) w = \sum_{\ell=n}^{L+1} w_{\ell}^{(n)}
	$.
	Using the triangle inequality and the Cauchy-Schwarz inequality, we obtain
	\[
	\begin{aligned}
	\widehat h_n^{-2}
	\|(Q_n-Q_{n-1}) w\|_{L_2(\partial\widehat\Omega)}^2
	&\le \widehat h_n^{-2} \left( \sum_{\ell=n}^{L+1} 
	\|w_{\ell}^{(n)}\|_{L_2(\partial\widehat\Omega)}\right)^2\\
	&\le \underbrace{\left( \sum_{\ell=n}^{L+1} 
		\frac{\widehat{h}_\ell^2}{\widehat{h}_n^2}\right)}_{\displaystyle \lesssim 1}
	\left( \sum_{\ell=n}^{L+1} \widehat h_\ell^{-2}
	\|w_{\ell}^{(n)}\|_{L_2(\partial\widehat\Omega)}^2\right).
	\end{aligned}
	\]
	By summing over $n$, we obtain using orthogonality
	\[
	\begin{aligned}
	\sum_{n=1}^{L+1}
	\widehat h_n^{-2}
	\|(Q_n-Q_{n-1}) w\|_{L_2(\partial\widehat\Omega)}^2
	&\lesssim \sum_{n=1}^{L+1} \sum_{\ell=n}^{L+1} \widehat h_\ell^{-2}
	\|w_{\ell}^{(n)}\|_{L_2(\partial\widehat\Omega)}^2
	=\sum_{\ell=1}^{L+1} \widehat h_\ell^{-2}
	\|w_{\ell}\|_{L_2(\partial\widehat\Omega)}^2
	,
	\end{aligned}
	\]
	which finishes the proof.
\end{proof}

\begin{lemma}\label{lem:qsum:h1}
	The estimate
	$
	\sum_{\ell=1}^{L+1} 
	\widehat h_\ell^{-2}
	\| (Q_\ell-Q_{\ell-1}) w \|_{L_2(\partial\widehat\Omega)}^2
	\lesssim \|w\|_{H^1(\partial\widehat\Omega)}^2
	$
	holds for all $w\in H^1(\partial\widehat\Omega)$.
\end{lemma}
\begin{proof}
	Using Lemma~\ref{lem:qsum:any} and by bounding the infimum by a particular
	decomposition, and using Lemma~\ref{lem:approx:pi}, $\Pi_0=0$ and $\widehat h_0\eqsim 1$ we obtain
	\[
	\begin{aligned}
	&	\sum_{\ell=1}^{L} 
	\widehat h_\ell^{-2}
	\| (Q_\ell-Q_{\ell-1}) w \|_{L_2(\partial\widehat\Omega)}^2
	\eqsim
	\inf_{w_\ell\in \mathcal W^{(\ell)},\,  w=\sum_{\ell=1}^{L+1} w_{\ell}}
	\sum_{\ell=1}^{L+1}
	\widehat h_{\ell}^{-2} \|w_{\ell}\|_{L_2(\partial\widehat\Omega)}^2\\
	&\qquad\le
	\sum_{\ell=1}^{L+1} 
	\widehat h_\ell^{-2}
	\| (\Pi_\ell-\Pi_{\ell-1}) w \|_{L_2(\partial\widehat\Omega)}^2
	\lesssim
	\sum_{\ell=1}^{L+1} 
	\| (\Pi_\ell-\Pi_{\ell-1}) w \|_{H^1(\partial\widehat\Omega)}^2
	= \|w\|_{H^1(\partial\widehat\Omega)}^2,
	\end{aligned}
	\]
	which finishes the proof.
\end{proof}

\begin{lemma}\label{lem:approx:qsum}
	The estimate
	$
	\sum_{\ell=1}^{L+1} \widehat h_\ell^{-1}
	\| (Q_\ell-Q_{\ell-1}) w \|_{L_2(\partial\widehat\Omega)}^2
	\lesssim \|w\|_{H^{1/2}(\partial\widehat\Omega)}^2
	$
	holds for all $w\in H^1(\partial\widehat\Omega)$.
\end{lemma}
\begin{proof}
	Orthogonality and Lemma~\ref{lem:qsum:h1} yield
	\[
	\| \mathcal Q w \|_{L_2(\partial\widehat\Omega)}^2
	=
	\| w \|_{L_2(\partial\widehat\Omega)}^2
	\quad
	\mbox{and}
	\quad
	\| \mathcal Q w \|_{L_h}^2
	\lesssim
	\| w \|_{H^1(\partial\widehat\Omega)}^2,
	\]
	where
	\[
	\mathcal Q := \left(
	\begin{array}{c}
	Q_1 - Q_0 \\
	Q_2 - Q_1 \\
	\vdots\\
	Q_{L+1} - Q_L
	\end{array}
	\right),
	\quad\mbox{and}\quad
	\| (w_1,\cdots,w_{L+1}) \|_{L_h}^2 := \sum_{\ell=1}^{L+1} h_\ell^{-2} \|w_{\ell}\|_{L_2(\partial\widehat\Omega)}^2.
	\]
	Using Hilbert space interpolation theory, cf. Theorems~7.23 and 7.31 in Ref.~\refcite{AdamsFournier:2003},
	we immediately obtain the desired result.
\end{proof}

As a next step, we define extension operators that extend the solution from one edge
into the interior. Let $\widehat \Gamma^{(\delta)} := \{ \gamma(s)\;:\; s\in (\delta-1,\delta)\}$
for $\delta=1,2,3,4$
be the four sides of $\widehat \Omega$. We define the extension operator $\mathcal E^{(\ell,\widehat \Gamma^{(1)})}: \mathcal W \rightarrow \mathcal V$ as follows:
\[
(\mathcal E^{(\ell,\widehat \Gamma^{(1)})} w)(x,y) := 
w(0,y) \underbrace{ \max\{0,x/\eta^{(\ell,\widehat \Gamma^{(1)})}\}^p }_{\displaystyle \theta^{(\ell,\widehat \Gamma^{(1)})}(x):=},
\]
where $\eta^{(\ell,\widehat \Gamma^{(1)})}:=
\max \{ \zeta^{(k,2)}_i \,:\, \zeta^{(k,2)}_i \le \widehat{h}_\ell \}$
is the largest breakpoint which is smaller than or equal to
$\widehat{h}_\ell$. Note that~\eqref{eq:gridsize} implies that
$\eta^{(\ell,\widehat \Gamma^{(1)})}>0$ and that
\begin{equation}\label{eq:etasize}
\eta^{(\ell,\widehat \Gamma^{(1)})}\eqsim \widehat{h}_\ell.
\end{equation}
Note that by construction
\[
\begin{aligned}
&	(\mathcal E^{(\ell,\widehat \Gamma^{(1)})} w)|_{\widehat \Gamma^{(1)}} = w,
\quad
(\mathcal E^{(\ell,\widehat \Gamma^{(1)})} w)|_{\widehat \Gamma^{(2)}} = w(0,1)\,\theta^{(\ell,\widehat \Gamma^{(1)})}(x),
\\
&	(\mathcal E^{(\ell,\widehat \Gamma^{(1)})} w)|_{\widehat \Gamma^{(3)}} = 0,
\quad 
(\mathcal E^{(\ell,\widehat \Gamma^{(1)})} w)|_{\widehat \Gamma^{(4)}} = w(0,0)\,\theta^{(\ell,\widehat \Gamma^{(1)})}(x) 		
\end{aligned}
\]
for all $w\in\mathcal W$.
For the other edges we define the extension
operator $\mathcal E^{(\ell,\widehat \Gamma^{(\delta)})}$ and the function $\theta^{(\ell,\widehat \Gamma^{(\delta)})}$ analogously.

Let $\widehat{\textbf x}^{(\delta)}:=\gamma(\delta-1)$ for $\delta=1,2,3,4$ be the vertices of
$\widehat \Omega = (0,1)^2$. Consider the vertex
$\widehat{\textbf x}^{(1)}=(0,0)^\top$. The adjacent edges are $\widehat \Gamma^{(1)}$
and $\widehat \Gamma^{(4)}$. We define a vertex extension operator
$\mathcal E^{(\ell,\widehat{\textbf x}^{(1)})}: \mathcal W \rightarrow \mathcal V$ as follows:
\[
(\mathcal E^{(\ell,\widehat{\textbf x}^{(1)})} w)(x,y)
:= w(0,0)
\theta^{(\ell,\widehat{\Gamma}^{(1)})}(x)
\theta^{(\ell,\widehat{\Gamma}^{(4)})}(y).
\]
For any other vertices $\widehat{\textbf x}^{(\delta)}$, we define $\mathcal E^{(\ell,\widehat{\textbf x}^{(\delta)})}$
analogously. Let
$\mathcal E: \mathcal W \rightarrow \mathcal V$ be an overall
extension operator, defined as follows:
\begin{equation}\label{eq:e:def}
\mathcal E := \sum_{\ell=1}^{L+1} \mathcal E^{(\ell)} ( Q_\ell - Q_{\ell-1} ),
\quad\mbox{where}\quad
\mathcal E^{(\ell)}:=
\sum_{\delta=1}^4 \mathcal E^{(\ell,\widehat \Gamma^{(\delta)})}
-
\sum_{\delta=1}^4 \mathcal E^{(\ell,\widehat{\textbf x}^{(\delta)})}.
\end{equation}
By construction, we have
\begin{equation}\label{eq:e:is:extension}
(\mathcal E^{(\ell)} w)|_{\partial\widehat\Omega} = w
\quad\mbox{and}\quad
(\mathcal E w)|_{\partial\widehat\Omega} = \sum_{\ell=1}^{L+1} ( Q_\ell - Q_{\ell-1} ) w
= ( Q_{L+1} - Q_0 ) w = w.
\end{equation}

As a next step, we estimate the $H^1$-seminorm of the extension. Here, we estimate the
constituent parts of $\mathcal E^{(\ell)}$ separately.
\begin{lemma}\label{lem:E:gamma}
	The estimate
	\[
	\begin{aligned}
	(\mathcal E^{(\ell,\widehat \Gamma^{(\delta)})} w,
	\mathcal E^{(n,\widehat \Gamma^{(\delta)})} q)_{H^1(\widehat\Omega)}^2
	\lesssim &
	\sqrt{\frac{\min \{\widehat{h}_\ell,\widehat{h}_n\}}{\max\{\widehat{h}_\ell,\widehat{h}_n\}}}
	\big( p^{-1}\widehat{h}_\ell  |w|_{H^1(\widehat \Gamma^{(\delta)})}^2+ p^{-1}\widehat{h}_n  |q|_{H^1(\widehat \Gamma^{(\delta)})}^2
	\\&\qquad
	+p \widehat{h}_\ell^{-1} \|w\|_{L_2(\widehat \Gamma^{(\delta)})}^2+ p \widehat{h}_n^{-1}\|q\|_{L_2(\widehat \Gamma^{(\delta)})}^2 \big)
	\end{aligned}
	\]
	holds for all $w\in \mathcal W^{(\ell)}$, $q\in \mathcal W^{(n)}$ with $\ell,n=1,\ldots,L+1$
	and $\delta=1,\ldots,4$.
\end{lemma}
\begin{proof}
	Observe that the definition of the extension operator and the tensor-product structure
	of the domain immediately imply
	\[
	\begin{aligned}
	&(\mathcal E^{(\ell,\widehat \Gamma^{(\delta)})} w,
	\mathcal E^{(n,\widehat \Gamma^{(\delta)})} q)_{H^1(\widehat\Omega)}
	\\&\quad=
	(\theta^{(\ell,\widehat \Gamma^{(\delta)})},\theta^{(n,\widehat \Gamma^{(\delta)})})_{L_2(0,1)}
	(w,q)_{H^1(\widehat \Gamma^{(\delta)})}
	+
	(\theta^{(\ell,\widehat \Gamma^{(\delta)})},\theta^{(n,\widehat \Gamma^{(\delta)})})_{H^1(0,1)}
	(w,q)_{L_2(\widehat \Gamma^{(\delta)})}.
	\end{aligned}
	\]
	By computing the corresponding integrals, we immediately obtain
	\begin{equation}\label{eq:lem:E:gamma:proof}
	\begin{aligned}
	&|(\mathcal E^{(\ell,\widehat \Gamma^{(\delta)})} w,
	\mathcal E^{(n,\widehat \Gamma^{(\delta)})} q)_{H^1(\widehat\Omega)}|
	\\
	&\qquad \eqsim
	p^{-1}\min\{\widehat{h}_\ell,\widehat{h}_n\}
	|(w,q)_{H^1(\widehat \Gamma^{(\delta)})}|
	+
	p(\max\{\widehat{h}_\ell,\widehat{h}_n\})^{-1}
	|(w,q)_{L_2(\widehat \Gamma^{(\delta)})}|.
	\end{aligned}
	\end{equation}
	Using the Cauchy-Schwarz inequality and the geometric-arithmetic mean inequality,
	we obtain
	\[
	(w,q)_{H^1(\widehat \Gamma^{(\delta)})}
	\le
	\sqrt{\frac{1}{\widehat{h}_\ell\widehat{h}_n}}
	(
	\widehat{h}_\ell
	|w|_{H^1(\widehat \Gamma^{(\delta)})}^2
	+
	\widehat{h}_n
	|q|_{H^1(\widehat \Gamma^{(\delta)})}^2
	)
	\]
	and
	\[
	(w,q)_{L_2(\widehat \Gamma^{(\delta)})}
	\le
	\sqrt{\widehat{h}_\ell\widehat{h}_n}
	(
	\widehat{h}_\ell^{-1}
	\|w\|_{L_2(\widehat \Gamma^{(\delta)})}^2
	+
	\widehat{h}_n^{-1}
	\|q\|_{L_2(\widehat \Gamma^{(\delta)})}^2
	).
	\]
	By plugging these estimates into~\eqref{eq:lem:E:gamma:proof}, we 
	obtain the desired result.
\end{proof}
The following lemma allows to estimate the vertex extensions by the function values on
one of the adjacent edges.
\begin{lemma}\label{lem:E:V}
	The estimate
	\[
	\begin{aligned}
	(\mathcal E^{(\ell,\widehat{\normalfont \textbf x}^{(\delta)})} w,
	\mathcal E^{(n,\widehat{\normalfont \textbf x}^{(\delta)})} q)_{H^1(\widehat\Omega)}^2
	\lesssim &
	\frac{\min \{\widehat{h}_\ell,\widehat{h}_n\}}{\max\{\widehat{h}_\ell,\widehat{h}_n\}}
	\big( p^{-1}\widehat{h}_\ell  |w|_{H^1(\widehat \Gamma^{(\delta)})}^2+ p^{-1}\widehat{h}_n  |q|_{H^1(\widehat \Gamma^{(\delta)})}^2
	\\&\qquad
	+p \widehat{h}_\ell^{-1} \|w\|_{L_2(\widehat \Gamma^{(\delta)})}^2+ p \widehat{h}_n^{-1}\|q\|_{L_2(\widehat \Gamma^{(\delta)})}^2 \big)
	\end{aligned}
	\]
	holds for all $w\in \mathcal W^{(\ell)}$, $q\in \mathcal W^{(n)}$ with $\ell,n=1,\ldots,L+1$
	and $\delta=1,\ldots,4$.
\end{lemma}
\begin{proof}
	Let $\delta$ be arbitrary but fixed and observe that the edges adjacent to
	$\widehat{\textbf x}^{(\delta)}$ are $\widehat \Gamma^{(\delta-1)}$ and $\widehat \Gamma^{(\delta)}$,
	where we make use of $\widehat \Gamma^{(0)}:=\widehat \Gamma^{(4)}$. Direct computations
	yield
	\begin{equation}\label{eq:lem:E:V:proof}
	\begin{aligned}
	&|(\mathcal E^{(\ell,\widehat{\textbf x}^{(\delta)})} w,
	\mathcal E^{(n,\widehat{\textbf x}^{(\delta)})} q)_{H^1(\widehat\Omega)}^2|\\
	&\qquad=
	\Big(
	(\theta^{(\ell,\widehat \Gamma^{(\delta-1)})},\theta^{(n,\widehat \Gamma^{(\delta-1)})})_{L_2(0,1)}
	(\theta^{(\ell,\widehat \Gamma^{(\delta)})},\theta^{(n,\widehat \Gamma^{(\delta)})})_{H^1(0,1)}
	\\&\qquad\qquad+
	(\theta^{(\ell,\widehat \Gamma^{(\delta-1)})},\theta^{(n,\widehat \Gamma^{(\delta-1)})})_{H^1(0,1)}
	(\theta^{(\ell,\widehat \Gamma^{(\delta)})},\theta^{(n,\widehat \Gamma^{(\delta)})})_{L_2(0,1)}
	\Big)
	|w(\widehat{\textbf x}^{(\delta)})|\,|q(\widehat{\textbf x}^{(\delta)})|\\
	&\qquad\eqsim 
	\frac{\min \{\widehat{h}_\ell,\widehat{h}_n\}}{\max\{\widehat{h}_\ell,\widehat{h}_n\}}
	|w(\widehat{\textbf x}^{(\delta)})|\,|q(\widehat{\textbf x}^{(\delta)})|
	\le
	\frac{\min \{\widehat{h}_\ell,\widehat{h}_n\}}{\max\{\widehat{h}_\ell,\widehat{h}_n\}}
	(|w(\widehat{\textbf x}^{(\delta)})|^2+|q(\widehat{\textbf x}^{(\delta)})|^2)
	\end{aligned}
	\end{equation}
	Using standard estimates, we obtain
	\[
	|w(\widehat{\textbf x}^{(\delta)})|^2 \lesssim \|w\|_{L_2(\widehat \Gamma^{(\delta)})}\|w\|_{H^1(\widehat \Gamma^{(\delta)})}
	\lesssim 
	p\widehat{h}_\ell^{-1} \|w\|_{L_2(\widehat \Gamma^{(\delta)})}^2
	+p^{-1} \widehat{h}_\ell|w|_{H^1(\widehat \Gamma^{(\delta)})}^2
	\]
	and an analogous statement for $q$. By plugging these results into~\eqref{eq:lem:E:V:proof},
	we obtain the desired result.
\end{proof}

\begin{lemma}\label{lem:hhsum}
	The relation 
	$\sum_{\ell=1}^{L+1} \frac{\min \{\widehat{h}_\ell,\widehat{h}_n\}}{\max\{\widehat{h}_\ell,\widehat{h}_n\}}
	\eqsim \sum_{\ell=1}^{L+1} \sqrt{\frac{\min \{\widehat{h}_\ell,\widehat{h}_n\}}{\max\{\widehat{h}_\ell,\widehat{h}_n\}}} \eqsim 1$
	holds for all $n=1,\ldots,{L+1}$.
\end{lemma}
\begin{proof}
	By splitting the sum and using the summation formula for the geometric series, we obtain
	\begin{align*}
		\sum_{\ell=1}^{L+1} \frac{\min \{\widehat{h}_\ell,\widehat{h}_n\}}{\max\{\widehat{h}_\ell,\widehat{h}_n\}}
		= \sum_{\ell=1}^n \frac{\widehat{h}_n}{\widehat{h}_\ell}
		+\sum_{\ell=n+1}^{L+1} \frac{\widehat{h}_\ell}{\widehat{h}_n}
		= \sum_{\ell=1}^n 4^{\ell-n} + \sum_{\ell=n+1}^{L+1}4^{n-\ell} \eqsim 1.
	\end{align*}
	The proof of $\sum_{\ell=1}^{L+1} \sqrt{\frac{\min \{\widehat{h}_\ell,\widehat{h}_n\}}{\max\{\widehat{h}_\ell,\widehat{h}_n\}}}\eqsim 1$ can be done analogously.
\end{proof}
Using these estimates, we are now able to show that the overall extension operator $\mathcal E$
is bounded as follows.
\begin{lemma}\label{lem:e:estimate}
	The estimate
	$
	|\mathcal E w|_{H^1(\widehat\Omega)}^2
	\lesssim p 
	|w|_{H^{1/2}(\partial\widehat\Omega)}^2
	$
	holds for all for all $w\in \mathcal W$.
\end{lemma}
\begin{proof}
	Let $w\in \mathcal W$ be arbitrary but fixed and define
	$w_\ell := (Q_\ell -Q_{\ell-1})w$ for $\ell=1,\ldots,L+1$.
	Observe that the definition and the triangle inequality yield
	\[
	|w_{L+1}|_{H^1(\partial\widehat\Omega)}^2
	=
	|(I-Q_L)w|_{H^1(\partial\widehat\Omega)}^2
	\le
	|(I-\Pi_L)w|_{H^1(\partial\widehat\Omega)}^2
	+ |(\Pi_L-Q_L)w|_{H^1(\partial\widehat\Omega)}^2,
	\]
	where $\Pi_L$ is the $H^1$-orthogonal projection.  Using
	stability and Lemma~\ref{lem:inverse:half}, we obtain
	$\|(I-\Pi_L) w\|_{H^1(\partial\widehat\Omega)}^2	 \lesssim p^2 \widehat h^{-1}
	\|w\|_{H^{1/2}(\partial\widehat\Omega)}^2$. Using this and 
	Lemma~\ref{lem:inverse:q} and the triangle inequality, we obtain
	\[
	|w_{L+1}|_{H^1(\partial\widehat\Omega)}^2
	\lesssim
	p^2 \widehat h^{-1} \|w\|_{H^{1/2}(\partial\widehat\Omega)}^2
	+ \widehat h^{-2} \|(I-\Pi_L)w\|_{L_2(\partial\widehat\Omega)}^2
	+ \widehat h^{-2} \|(I-Q_L)w\|_{L_2(\partial\widehat\Omega)}^2.
	\]
	Using Lemmas~\ref{lem:approx:q} and~\ref{lem:approx:pi}, we further obtain
	\begin{equation}\label{eq:w:ell:estim}
	|w_{L+1}|_{H^1(\partial\widehat\Omega)}^2
	\lesssim
	p^2 \widehat h^{-1} \|w\|_{H^{1/2}(\partial\widehat\Omega)}^2.
	\end{equation}

	Observe that \eqref{eq:e:def},
	the triangle inequality yield,
	and Assumption~\ref{ass:neighbors}
	\begin{equation}\label{eq:lem:fin:1}
	|{\mathcal E} w|_{H^1(\widehat\Omega)}^2
	\lesssim
	\sum_{\delta=1}^4
	\left|\sum_{\ell=1}^{L+1} \mathcal E^{(\ell,\widehat \Gamma^{(\delta)})} w_\ell
	\right|_{H^1(\widehat\Omega)}^2		
	+ \sum_{\delta=1}^4
	\left|\sum_{\ell=1}^{L+1} \mathcal E^{(\ell,\widehat{\textbf x}^{(\delta)})} w_\ell
	\right|_{H^1(\widehat\Omega)}^2	.
	\end{equation}
	Using Lemmas~\ref{lem:E:gamma} and~\ref{lem:hhsum}, we obtain
	\begin{align*}
		\left|\sum_{\ell=1}^{L+1} \mathcal E^{(\ell,\widehat \Gamma^{(\delta)})} w_\ell
		\right|_{H^1(\widehat\Omega)}^2	
		&= \sum_{\ell=1}^{L+1}\sum_{n=1}^{L+1}
		\left( \mathcal E^{(\ell,\widehat \Gamma^{(\delta)})} w_\ell
		, \mathcal E^{(n,\widehat \Gamma^{(\delta)})} w_n
		\right)_{H^1(\widehat\Omega)}\\
		&\lesssim  p^{-1} \sum_{\ell=1}^{L+1} \widehat h_\ell  |w_\ell |_{H^1(\partial\widehat\Omega)}^2+
		p \sum_{\ell=1}^{L+1} \widehat h_\ell^{-1} \|w_\ell\|_{L_2(\partial\widehat\Omega)}^2.
	\end{align*}
	Analogously, we obtain using
	Lemmas~\ref{lem:E:V} and~\ref{lem:hhsum}
	\[
	\left|\sum_{\ell=1}^{L+1} \mathcal E^{(\ell,\widehat{\textbf x}^{(\delta)})} w_\ell
	\right|_{H^1(\widehat\Omega)}^2	
	\lesssim  p^{-1} \sum_{\ell=1}^{L+1} \widehat h_\ell  |w_\ell |_{H^1(\partial\widehat\Omega)}^2+
	p \sum_{\ell=1}^{L+1} \widehat h_\ell^{-1} \|w_\ell\|_{L_2(\partial\widehat\Omega)}^2.
	\]
	By plugging these two estimates into~\eqref{eq:lem:fin:1}, we obtain
	\[
	|{\mathcal E} w|_{H^1(\widehat\Omega)}^2
	\lesssim  p^{-1} \sum_{\ell=1}^{L+1} \widehat h_\ell  |w_\ell |_{H^1(\partial\widehat\Omega)}^2+
	p \sum_{\ell=1}^{L+1} \widehat h_\ell^{-1} \|w_\ell\|_{L_2(\partial\widehat\Omega)}^2.
	\]
	Using~\eqref{eq:w:ell:estim} and Lemma~\ref{lem:inverse:q}, we obtain further
	\[
	\begin{aligned}
	|{\mathcal E} w|_{H^1(\widehat\Omega)}^2
	&\lesssim p \|w\|_{H^{1/2}(\partial\widehat\Omega)}^2
	+ p^{-1} \sum_{\ell=1}^{L} \widehat h_\ell^{-1} \|w_\ell\|_{L_2(\partial\widehat\Omega)}^2
	+ p \sum_{\ell=1}^{L+1} \widehat h_\ell^{-1} \|w_\ell\|_{L_2(\partial\widehat\Omega)}^2\\
	&\lesssim  p \|w\|_{H^{1/2}(\partial\widehat\Omega)}^2
	+ p \sum_{\ell=1}^{L+1} \widehat h_\ell^{-1} \|w_\ell\|_{L_2(\partial\widehat\Omega)}^2.
	\end{aligned}
	\]
	Using $w_\ell=(Q_\ell-Q_{\ell-1})w$ and Lemma~\ref{lem:approx:qsum}, we obtain further
	\[
	|{\mathcal E} w|_{H^1(\widehat\Omega)}^2
	\lesssim  p \|w\|_{H^{1/2}(\partial\widehat\Omega)}^2.
	\]
	A Poincar{\'e} type argument finishes the proof.
\end{proof}

Since we have an estimate for the extension operator $\mathcal E$, we can also give a corresponding
estimate for the discrete harmonic extension $\mathcal H_h$. Before we do this, we state
a standard result on the equivalence of the norms between physical domain and parameter domain.\begin{lemma}\label{lem:geoequiv}
	For all patches $k=1,\ldots,K$, we have
	\begin{itemize}
		\item
		$|u|_{H^s(\Omega^{(k)})} \eqsim (H^{(k)})^{1-s} | u\circ G_k |_{H^s(\widehat\Omega)}$
		for all $u\in H^s(\Omega^{(k)})$ and $s\in \{0,1\}$, and
		\item
		$|w|_{H^s(\partial\Omega^{(k)})} \eqsim (H^{(k)})^{1/2-s} | w\circ G_k |_{H^s(\partial\widehat\Omega)}$ for all $w\in H^{1/2}(\partial \Omega^{(k)})$ and $s\in~\{0,1/2,1\}$,
	\end{itemize}
	where we use the notation $H^0:=L_2$. The same holds if
	$\partial\Omega^{(k)}$ and $\partial\widehat \Omega$ are replaced by $\Gamma^{(k,\ell)}$
	and $\widehat\Gamma^{(k,\ell)}$, respectively.
\end{lemma}
\begin{proof}
	The results for $s\in\{0,1\}$ directly follow from Assumption~\ref{ass:nabla}
	and the chain rule for differentiation and the substitution rule, see, e.g,
	Lemma~3.5 in Ref.~\refcite{Bazilevs}.
	The results for $s=1/2$ are then obtained by Hilbert space interpolation theory,
	cf. Theorems~7.23 and 7.31 in Ref.~\refcite{AdamsFournier:2003}.
\end{proof}

Finally, we can state the main theorem of this section.
\begin{theorem}\label{thrm:discrharm:ext}
	The estimate
	\[
	|w|_{H^{1/2}(\partial\Omega^{(k)})}^2
	\lesssim
	|\mathcal H_h w|_{H^1(\Omega^{(k)})}^2
	\lesssim p |w|_{H^{1/2}(\partial\Omega^{(k)})}^2
	\]
	holds for all $w \in W^{(k)}$.
\end{theorem}
\begin{proof}
	Let $\mathcal H:H^{1/2}(\partial\Omega^{(k)})
	\rightarrow H^1(\Omega^{(k)})$ be the harmonic extension. Since the $H^1$-seminorm of the
	harmonic extension is equivalent to the $H^{1/2}$-seminorm on the boundary and
	since the harmonic extension minimizes over $H^1(\Omega^{(k)})\supset V^{(k)}$,
	we have
	\[
	|w|_{H^{1/2}(\partial\Omega^{(k)})}^2
	\lesssim
	|\mathcal H w|_{H^1(\Omega^{(k)})}^2
	\le
	|\mathcal H_h w|_{H^1(\Omega^{(k)})}^2,
	\]
	which shows the first part of the estimate. 	Since the discrete harmonic extension
	minimizes over $V^{(k)}$, we obtain using Lemma~\ref{lem:geoequiv}
	\[
	|\mathcal H_h w|_{H^1(\Omega^{(k)})}^2
	= \inf_{v\in V^{(k)}\;:\;v|_{\partial \Omega^{(k)}}=w} |v|_{H^1(\Omega^{(k)})}^2
	\eqsim \inf_{v\in \widehat V^{(k)}\;:\;v|_{\partial\widehat\Omega}=w\circ G_k}
	|v|_{H^1(\widehat \Omega)}^2.
	\]
	Since the operator $\mathcal E$ preserves the Dirichlet boundary conditions, we have that
	$\mathcal E (w\circ G_k) \in \widehat V^{(k)}$ and thus
	we obtain using Lemmas~\ref{lem:e:estimate} and~\ref{lem:geoequiv}
	\[
	|\mathcal H_h w|_{H^1(\widehat\Omega)}^2
	\le
	|\mathcal E (w \circ G_k)|_{H^1(\widehat\Omega)}^2
	\lesssim
	p |w \circ G_k|_{H^{1/2}(\partial\widehat\Omega)}^2
	\eqsim
	p |w|_{H^{1/2}(\partial\Omega^{(k)})}^2
	,
	\]
	which shows the second part of the estimate. 
\end{proof}

\subsection{An embedding result}\label{sec:4:2}

The following Lemma shows that we are able to bound the function values of a spline
function from above using the $H^1$-norm. If we consider the physical domain, the
additional scaling factor is such that it can be eliminated using the Poincar{\'e}
inequality.

\begin{lemma}\label{lem:param:point:values}
	The estimates
	\[
	\sup_{x\in \overline{\widehat\Omega}} |\widehat u(x)|^2
	\lesssim \Lambda \|\widehat u\|_{H^1(\widehat\Omega)}^2
	\;\mbox{and}\;
	\sup_{x\in \overline{\Omega^{(k)}}} |u(x)|^2
	\lesssim \Lambda \big(|u|_{H^1(\Omega^{(k)})}^2
	+ (H^{(k)})^{-2} \|u\|_{L_2(\Omega^{(k)})}^2 \big),
	\]
	where $\Lambda:= 1+\log p+\log \frac{H^{(k)}}{h^{(k)}}$,
	hold for all $\widehat u \in \widehat V^{(k)}$ and $u \in V^{(k)}$	.
\end{lemma}
\begin{proof}
	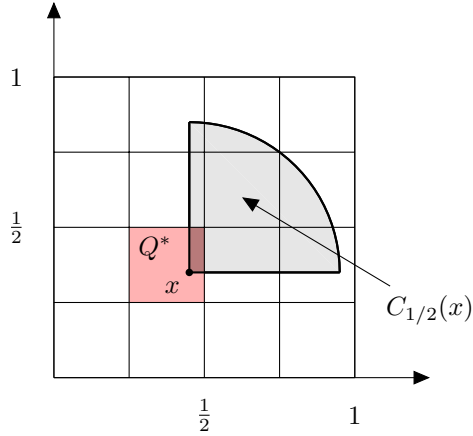
\begin{figure}[h!]
		\begin{minipage}{1\linewidth}
			\centering
			\begin{tikzpicture}
			\draw [black,thick,domain=0:90,fill=gray!20] plot ({1.8+2*cos(\x)}, {1.4+2*sin(\x)});
			\draw [black, thick,fill=gray!20] (1.8,3.4) -- (1.8,1.4) -- (3.8,1.4) {};
			\fill[red!30] (1,1) rectangle (2,2);
			\fill[red!30!gray!80] (1.8,1.4) rectangle (2,2);
			\draw[-triangle 45] (0,0) -- (5,0) node at (4,-0.5) {$1$} node at (2, -0.5) {$\frac{1}{2}$};
			\draw[-triangle 45] (0,0) -- (0,5) node at (-0.5, 4) {$1$} node at (-0.5, 2) {$\frac{1}{2}$};
			
			\draw (4,0) -- (4,4) -- (0,4) node at (4.2,4.2) {};%
			\draw[step = 1, black, thin] (0,0) grid (4,4);
			
			\draw (1.8,1.4) node[circle, fill, inner sep = 1pt] (M) {};
			\node[anchor = north east] at (M) {$x$};
			\node[anchor = north west] at (1, 2) {$Q^*$};
			
			\draw [black,thick,domain=0:90] plot ({1.8+2*cos(\x)}, {1.4+2*sin(\x)});
			\draw [black, thick] (1.8,3.4) -- (1.8,1.4) -- (3.8,1.4) {};
			
			\node at (5, 0.9) (C) {$C_{1/2}(x)$};
			\draw[triangle 45-] (2.5, 2.4) -- (C);
			\end{tikzpicture}
			\captionof{figure}{The chosen point $x$ with its cone of radius $r=1/2$}
			\label{fig:4.4caseA}
		\end{minipage}%
	\end{figure}

	Let $\widehat u\in \widehat{V}^{(k)}$ and  $x\in\overline{\widehat \Omega}$ be arbitrary but fixed.
	We assume without loss of generality that
	\begin{equation}\label{eq:x:wlog}
	x \in [0,1/2]^2.
	\end{equation}
	
	The domain $[0,1)^2$ is composed of elements $Q^{(i)} := [\underline{q}_1^{(i)},\overline{q}_1^{(i)})
	\times [\underline{q}_2^{(i)},\overline{q}_2^{(i)})$ on which the
	function $u$ is polynomial. Let $Q^*:= [\underline{q}_1^*,\overline{q}_1^*)
	\times [\underline{q}_2^*,\overline{q}_2^*)$ be the element containing $x$. 
	Using Assumption~\ref{ass:quasiuniform}, we obtain
	\begin{equation}\label{eq:qgrid}
	\widehat{h}^{(k)}\eqsim \widehat h := \max_{\delta=1,2} \overline{q}_\delta^* - \underline{q}_\delta^*.
	\end{equation}
	For any $r>0$, let $C_r(x)$ be the cone with vertex $x=(x_1,x_2)$, defined by
	\[
	C_r(x) :=
	\{
	\xi=(\xi_1,\xi_2) \in \mathbb{R}^2 \,:\,
	\|\xi-x\|_{\ell^2} \le r, \;
	\xi_1 \ge x_1, \;
	\xi_2 \ge x_2
	\}.
	\]
	The assumption~\eqref{eq:x:wlog} yields $C_{1/2}(x) \subset [0,1]^2$.
	Consider the case that $C_\epsilon(x) \subset \overline{Q^*}$, where
	$\epsilon := 4^{-1} p^{-4} \widehat{h}^{3/2}$, first.
	Using the notation
	\[
	v(r,\theta) :=
	\widehat u( x_1 + r \sin \theta, x_2 + r \cos \theta)
	\]
	and the fundamental theorem of calculus, we obtain 
	\[
	\widehat u(x) = v(0,\theta) = v(r,\theta)
	- \int_{0}^{r} \frac{\partial v}{\partial \rho}(\rho, \theta) \; \textit{d}\rho
	\]
	for all $r \in (0,1/2]$ and all $\theta \in [0,\pi/2]$.
	By splitting up the integral and using Young's inequality, we obtain
	\begin{equation}
	\label{eq:first:splitting}
	\widehat u(x)^2
	\lesssim v(r,\theta)^2
	+\left(\int_{0}^{\epsilon} \frac{\partial v}{\partial \rho}(\rho, \theta) \; \textit{d} \rho\right)^2
	+\left(
	\int_{\epsilon}^{r} \frac{\partial v}{\partial \rho}(\rho, \theta) \; \textit{d} \rho \right)^2
	\end{equation}
	for all $r \in (1/4,1/2)\subset(\epsilon,1/2)$.
	Using Theorem~4.76, eq. (4.6.2) and (4.6.1) in Ref.~\refcite{Schwab}, we estimate the first integral as follows
	\begin{equation}
	\label{eq:first:integral}
	\int_{0}^{\epsilon} \frac{\partial v}{\partial \rho}(\rho, \theta) \; \textit{d} \rho
	\leq
	\epsilon \Big\|\frac{\partial v}{\partial \rho}\Big\|_{L_\infty((0,\epsilon)\times(0,\pi/2))}
	\leq
	\epsilon |\widehat u|_{W^1_\infty(Q^*)}
	\lesssim p^4 \widehat{h}^{-3/2} \epsilon \|\widehat u\|_{L_2(Q^*)}.
	\end{equation}
	The second inequality is estimated using the
	Cauchy-Schwarz inequality as follows
	\begin{align*}
		\int_{\epsilon}^{r} \frac{\partial v}{\partial \rho}(\rho, \theta) \; \textit{d} \rho 
		& =
		\int_{\epsilon}^{r}\rho^{-1/2} \left( \frac{\partial v}{\partial \rho}(\rho, \theta) \rho^{1/2}\right) 	\; \textit{d} \rho \\
		& \leq
		(\log r - \log \epsilon)^{1/2}
		\left(
		\int_{\epsilon}^{r} \frac{\partial v}{\partial \rho}(\rho, \theta)^2 \rho \; \textit{d} \rho 
		\right)^{1/2}.
	\end{align*}
	By transforming the integral back to Cartesian coordinates, we obtain further
	\begin{align}
		\label{eq:second:integral}
		\int_{\epsilon}^{r} \frac{\partial v}{\partial \rho}(\rho, \theta) \; \textit{d} \rho 
		\leq
		(\log r - \log \epsilon)^{1/2}
		|\widehat u|_{H^1(C_r(x)\backslash C_\epsilon(x))}
		\lesssim
		(- \log \epsilon)^{1/2}
		|\widehat u|_{H^1(C_{1/2}(x)\backslash C_\epsilon(x))}
	\end{align}
	for all $r\in(1/4,1/2)$.
	By plugging the estimates from 
	\eqref{eq:first:integral} and \eqref{eq:second:integral}
	into
	\eqref{eq:first:splitting}, we obtain
	\begin{align*}
		\widehat u(x)^2
		&\lesssim v(r,\theta)^2 + p^4 \widehat{h}^{-3/2} \epsilon   \|\widehat u\|_{L_2(Q^*)}^2
		- \log \epsilon |\widehat u|_{H^1(C_{1/2}(x)\backslash C_\epsilon(x))}^2\\
		&\lesssim v(r,\theta)^2 + (1+\log p- \log \widehat{h})\|\widehat u\|_{H^1(\widehat \Omega)}^2
	\end{align*}
	for all $r\in(1/4,1/2)$ and all $\theta \in [0,\pi/2]$.
	By taking the integral over the domain $C_{1/2}(x)\backslash C_{1/4}(x)$, we obtain
	\begin{align*}
		&\widehat  u(x)^2
		\eqsim
		\int_{C_{1/2}(x)\backslash C_{1/4}(x)} 
		\widehat u(x)^2
		\mbox{d}\xi
		=
		\int_0^{\pi/2} \int_{1/4}^{1/2}
		r\,\widehat u(x)^2\mbox{d}r\, \mbox{d}\theta\\
		&\qquad\lesssim 
		\int_0^{\pi/2} \int_{1/4}^{1/2} r v(r,\theta)^2 \mbox{d}r\, \mbox{d}\theta
		+  (1+\log p- \log \widehat{h})\|\widehat u\|_{H^1(\widehat \Omega)}^2 \int_0^{\pi/2} \int_{1/4}^{1/2} r\, \mbox{d}r\, \mbox{d}\theta  \\
		&\qquad  \lesssim  (1+\log p- \log \widehat{h})\|\widehat u\|_{H^1(\widehat \Omega)}^2,
	\end{align*}
	which shows the desired result.
	
	Now, we consider the case that $C_\epsilon \not\subset \overline{Q^*}$.
	In this case, we define
	\[
	\widetilde{x} := (
	\min\{ x_1, \overline{q}_1^*- \epsilon \},
	\min\{ x_2, \overline{q}_2^*- \epsilon \}
	)^\top
	\]
	and observe that~\eqref{eq:qgrid} yields that $\widetilde{x}\in Q^*$. By construction,
	$C_{\epsilon}(\widetilde{x})\subset \overline{Q^*}$.
	Thus, using the arguments above, we obtain
	\begin{equation}\label{eq:tilde:x}
	\widehat u(\widetilde{x})^2 \lesssim (1 + \log p- \log \widehat{h})\|\widehat u\|_{H^1(\widehat \Omega)}^2.
	\end{equation}
	Moreover, we observe that $\|x-\widetilde{x}\|_{\ell^2} \lesssim p^{-4} \widehat{h}^{3/2}$.
	Thus, we conclude using the fundamental theorem of calculus
	\[
	(\widehat u(x)-\widehat u(\widetilde{x}))^2 \le \|x-\widetilde{x}\|_{\ell^2} \|\widehat u\|_{W^1_\infty(Q^*)}^2
	\lesssim  p^{-4} \widehat{h}^{3/2} \|\widehat u\|_{W^1_\infty(Q^*)}^2.
	\]
	Using Theorem~4.76, eq. (4.6.2) and (4.6.1) in Ref.~\refcite{Schwab}, we further obtain 
	\[
	(\widehat u(x)-\widehat u(\widetilde{x}))^2 \lesssim  \|\widehat u\|_{L_2(Q^*)}^2.
	\]
	By combining this result with~\eqref{eq:tilde:x}, the triangle inequality
	and Young's inequality, we immediately obtain the first bound.
	The second statement follows directly using Lemma~\ref{lem:geoequiv}.
\end{proof}

\subsection{The tearing lemma}\label{sec:4:3}

The variation of a function $v$ over a domain $T$ is defined via
\[
|v|_{L_\infty^0(T)} := \inf_{c\in \mathbb R} \|v-c\|_{L_\infty(T)}
= \frac{1}{2}\big( \mbox{ess.sup}_{x\in T} v(x) - \mbox{ess.inf}_{x\in T} v(x)\big),
\]
where ess.sup and ess.inf are the essential supremum and infimum, respectively.
Thus, we have obviously
\begin{equation}\label{eq:variation}
|v|_{L_\infty^0(T_1\cup T_2)} 
\le |v|_{L_\infty^0(T_1)} 
+ |v|_{L_\infty^0(T_2)} 
\;
\mbox{and}
\;
\|v(x)-v\|_{L_\infty(T)} \le 2 |v|_{L_\infty(T)}
\mbox{ for all } x\in T.
\end{equation}

\begin{lemma}\label{lem:tearing}
	The estimates
	\begin{align*}
		& \sum_{\ell=1}^4
		|\widehat u|^2_{H^{1/2}(\widehat{\Gamma}^{(\ell)})}
		\le
		|\widehat u|^2_{H^{1/2}(\partial \widehat\Omega)}
		\lesssim 
		\sum_{\ell=1}^4
		\left(
		|\widehat u|^2_{H^{1/2}(\widehat{\Gamma}^{(\ell)})}
		+  \Lambda	|\widehat u|^2_{L_\infty^0(\widehat{\Gamma}^{(\ell)})}
		\right),  \\
		& \sum_{\ell=1}^4
		|u|^2_{H^{1/2}(\Gamma^{(k,\ell)})}
		\le
		|u|^2_{H^{1/2}(\partial \Omega^{(k)})}
		\lesssim 
		\sum_{\ell \in \mathcal N(k)}
		\left(
		|u|^2_{H^{1/2}(\Gamma^{(k,\ell)})}
		+  \Lambda 
		|u|^2_{L_\infty^0(\Gamma^{(k,\ell)})}
		\right)  
	\end{align*}
	hold for all $\widehat u \in \widehat V^{(k)}$ and $u \in V^{(k)}$,
	where $\Lambda:= 1+\log p+\log \frac{H^{(k)}}{h^{(k)}}$.
\end{lemma}
\begin{proof}
	We begin with the first statement.
	Obviously,
	\begin{equation}\label{eq:tearing:proof:1}
	| \widehat u |^2_{H^{1/2}(\partial \widehat{\Omega})}
	=
	\sum_{\ell=1}^{4} |\widehat u|^2_{H^{1/2}(\widehat{\Gamma}^{(\ell)})}
	+
	\underbrace{\sum_{\ell=1}^{4}
		\sum_{n \in \{1,\ldots,4\}\backslash \{\ell\}}^{}
		\int_{\widehat{\Gamma}^{(\ell)}}^{}\int_{\widehat{\Gamma}^{(n)}}^{} 
		\frac{|\widehat u(x) - \widehat u(y)|^2}{|x-y|^2} 
		\mathrm{d}x \mathrm{d}y}_{\displaystyle \ge 0},
	\end{equation}
	which immediately shows the first side of the desired inequality. For the second
	side, we have to estimate the double integral.
	Let us consider a term with $\ell \neq n$. The edges $\widehat{\Gamma}^{(\ell)}$
	are parameterized by the functions
	\[
	\gamma^{(1)}(t) := (0,t) ,\quad
	\gamma^{(2)}(t) := (1-t,1) ,\quad
	\gamma^{(3)}(t) := (1,1-t) ,\quad
	\gamma^{(4)}(t) := (t,0).
	\]
	(Note that this parameterization is different than that of~\eqref{eq:unroll}.)
	We define the functions $v := u\circ \gamma^{(\ell)}$ and $w := u\circ\gamma^{(n)}$.
	Simple calculations show that
	\[
	|s + t|^2 \lesssim |\gamma^{(m)}(s) - \gamma^{(n)}(t)|^2 
	\] 
	holds for all $m$ and $n$ with $m\ne n$.
	Thus, we obtain
	\[
	\int_{\widehat{\Gamma}^{(m)}}^{}\int_{\widehat{\Gamma}^{(n)}}^{} 
	\frac{|\widehat u(x) -\widehat u(y)|^2}{|x-y|^2} 
	\,\mathrm{d}x \,\mathrm{d}y
	\lesssim
	\int_{0}^{1}\int_{0}^{1} 
	\frac{|v(s) - w(t)|^2}{|s+t|^2} 
	\,\mathrm{d}s \,\mathrm{d}t 
	\]
	By adding a productive zero and using the triangle inequality, we obtain 
	\begin{equation}\label{eq:split0}
	\begin{aligned}
	&\int_{\widehat{\Gamma}^{(m)}}^{}\int_{\widehat{\Gamma}^{(n)}}^{} 
	\frac{|\widehat u(x) - \widehat u(y)|^2}{|x-y|^2} 
	\,\mathrm{d}x \,\mathrm{d}y
	\\
	&\qquad\qquad \lesssim
	\int_{0}^{1}\int_{0}^{1} 
	\frac{|v(s) - v(0)|^2}{|s+t|^2} 
	\,\mathrm{d}s \,\mathrm{d}t 
	+
	\int_{0}^{1}\int_{0}^{1} 
	\frac{|w(t) - w(0)|^2}{|s+t|^2} 
	\,\mathrm{d}s \,\mathrm{d}t
	\\
	&\qquad\qquad\eqsim
	\int_{0}^{1}
	\frac{|v(s) - v(0)|^2}{s} 
	\,\mathrm{d}s  
	+
	\int_{0}^{1}
	\frac{|w(t) - w(0)|^2}{t} 
	\,\mathrm{d}t
	\end{aligned}
	\end{equation}
	Consider the first of these summands.
	Let $\epsilon := p^{-4}(\widehat{h}^{(k)})^2$.
	By splitting up the integral, we obtain
	\begin{equation}\label{eq:split}
	\int_{0}^{1}
	\frac{|v(s) - v(0)|^2}{s} 
	\,\mathrm{d}s
	=
	\int_{0}^{\epsilon} 
	\frac{|v(s) - v(0)|^2}{s} 
	\,\mathrm{d}s 
	+
	\int_{\epsilon}^{1} 
	\frac{|v(s) - v(0)|^2}{s} 
	\,\mathrm{d}s .
	\end{equation}
	Using the fundamental theorem of calculus and the Cauchy-Schwarz
	inequality, we obtain 
	\[
	\int_{0}^{\epsilon} 
	\frac{|v(s) - v(0)|^2}{s} 
	\,\mathrm{d}s
	= 
	\int_{0}^{\epsilon}
	\frac{\left(
		\int_{0}^{s} v'(z) \; \text{d}z 
		\right)^2 }{s} \; \text{d}s 
	\le 
	\int_{0}^{\epsilon} \int_{0}^{s} v' (z)^2 \; \text{d}z \; \text{d}s
	\le \epsilon |v|_{H^1(0,1)}^2.
	\]
	Using a standard inverse estimate, cf. Theorem~4.76, eq. (4.6.5) in Ref.~\refcite{Schwab},
	we obtain further
	\begin{equation}\label{eq:split1}
	\int_{0}^{\epsilon} 
	\frac{|v(s) - v(0)|^2}{s} 
	\,\mathrm{d}s
	\lesssim 
	\epsilon p^4 (\widehat{h}^{(k)})^{-2} \inf_{c\in \mathbb R} \|v-c\|_{L_2(0,1)}^2
	\le 
	\epsilon p^4 (\widehat{h}^{(k)})^{-2}  |v|_{L_\infty^0(0,1)}^2.
	\end{equation}
	For the second integral in~\eqref{eq:split}, we obtain using~\eqref{eq:variation} that
	\begin{equation}\label{eq:split2}
	\int_{\epsilon}^{1}
	\frac{|v(s) - v(0)|^2}{s} 
	\mathrm{d}s
	\lesssim 
	\int_{\epsilon}^{1}
	\frac{1}{s} \; \text{d}s \; \| v(\cdot) - v(0) \|^2_{L_\infty(0,1)} 
	= -\log \epsilon \; |v|^2_{L_\infty^0(0,1)}.
	\end{equation}
	The combination of~\eqref{eq:split}, \eqref{eq:split1}, \eqref{eq:split2}, and
	and the definition of $\epsilon$	yields
	\[
	\int_{0}^{1}
	\frac{|v(s) - v(0)|^2}{s} 
	\,\mathrm{d}s
	\lesssim 	\left(1 + \log   p - \log   \widehat{h}^{(k)}\right) \; 
	|v|^2_{L_\infty^0(0,1)}.
	\]
	Since we can estimate the second integral in~\eqref{eq:split0} analogously, we have
	\begin{align*}
		&\int_{\widehat{\Gamma}^{(m)}}^{}\int_{\widehat{\Gamma}^{(n)}}^{} 
		\frac{|\widehat u(x) - \widehat u(y)|^2}{|x-y|^2} 
		\, \mathrm{d}x \, \mathrm{d}y  \lesssim
		\left(1 + \log   p - \log  \widehat{h}^{(k)}\right) \; 
		\left( 
		|\widehat u|^2_{L_\infty^0(\widehat\Gamma^{(m)})}
		+
		|\widehat u|^2_{L_\infty^0(\widehat\Gamma^{(n)})}
		\right).
	\end{align*}
	The combination of this estimate and~\eqref{eq:tearing:proof:1} shows the first statement.
	The second statement is obtained using Lemma~\ref{lem:geoequiv}.
\end{proof}

\subsection {Condition number estimate}\label{sec:4:fin}

In the following, we prove Theorem~\ref{thrm:fin}.
The idea of the proof is to use Theorem~22 in Ref.~\refcite{MandelDohrmannTezaur:2005a},
which states that 
\begin{equation}\label{eq:MandelDohrmannTezaur}
\kappa(M_{\mathrm{sD}} \, F) \le \sup_{w \in  \widetilde W }
\frac{ \| B_D^\top B \underline w \|_S^2 }{ \| \underline w \|_S^2 },
\end{equation}
where $\underline w$ is the coefficient vector associated to $w$.

\begin{lemma}\label{lem:bbt}
	Let $u=(u^{(1)},\cdots,u^{(K)})\in \widetilde W$ with coefficient vector $\underline u$
	and let $w=(w^{(1)},\cdots,w^{(K)})$ with coefficient vector $\underline w$
	be such that
	\[
	\underline w = B_D^\top B \underline u.
	\]
	Then, we have for each patch $\Omega^{(k)}$ and each edge $\Gamma^{(k,\ell)}$
	connecting the vertices ${\normalfont \textbf x}^{(k,\ell,1)}$ and ${\normalfont \textbf x}^{(k,\ell,2)}$
	\[
	\begin{aligned}
	&|w^{(k)}|_{H^{1/2}(\Gamma^{(k,\ell)})}^2
	\lesssim
	|u^{(k)}|_{H^{1/2}(\Gamma^{(k,\ell)})}^2
	+ |u^{(\ell)}|_{H^{1/2}(\Gamma^{(k,\ell)})}^2
	+ \Delta^{(k,\ell,1)} + \Delta^{(k,\ell,2)}, \\
	&|w^{(k)}|_{L_\infty^0(\Gamma^{(k,\ell)})}^2
	\lesssim
	|u^{(k)}|_{L_\infty^0(\Gamma^{(k,\ell)})}^2
	+ |u^{(\ell)}|_{L_\infty^0(\Gamma^{(k,\ell)})}^2
	+ \Delta^{(k,\ell,1)} + \Delta^{(k,\ell,2)},
	\end{aligned}
	\]
	where $\Delta^{(k,\ell,i)}=0$ for Algorithms~A and C and
	\[
	\Delta^{(k,\ell,i)}:=
	\sum_{j \in \mathcal{P}(\normalfont \textbf x^{(k,\ell,i)})}
	|u^{(k)}({\normalfont \textbf x}^{(k,\ell,i)})-u^{(j)}({\normalfont \textbf x}^{(k,\ell,i)})|^2 
	\]
	for Algorithm~B.
\end{lemma}
\begin{proof}
	First discuss the entries of the scaling matrix $D$, which was defined to
	be a diagonal matrix with coefficients $d_{i,i}$ for $i=1,\ldots,N_\Gamma$.
	$d_{i,i}$ is defined to be the number of Lagrange multipliers that act on
	the corresponding basis function. For each edge basis function, i.e., a
	basis function that is active on one edge and vanishes on all vertices, we
	have only one Lagrange multiplier. Thus, $d_{i,i}=1$ for the corresponding
	values of $i$.
	
	For the vertex basis functions, we have to distinguish based on the
	choice of the primal degrees of freedom. For the Algorithms~A and C,
	no Lagrange multiplier acts on the respective degree of freedom. Thus,
	we have $d_{i,i}=1$. For
	Algorithm~B, we have $d_{i,i}=|\mathcal P(\textbf x^{(i)})|-1$, where
	$\textbf{x}^{(i)}$ is the corresponding vertex, since we use a fully
	redundant scheme.
	
	Simple calculations yield for Algorithms~A and C that
	\[
	w^{(k)}|_{\Gamma^{(k,\ell)}} = u^{(k)}|_{\Gamma^{(k,\ell)}} - u^{(\ell)}|_{\Gamma^{(k,\ell)}} 
	- \sum_{i=1}^2 \theta^{(k,\ell,i)}
	\big( u^{(k)}(\textbf x^{(k,\ell,i)}) - u^{(\ell)}(\textbf x^{(k,\ell,i)}) \big).
	\]
	where $\theta^{(k,\ell,i)}$ is the basis function in $\Phi^{(k)}$
	such that $\theta^{(k,\ell,i)}(\textbf x^{(k,\ell,i)})=1$.
	Since $u$ satisfies the primal constraints, we have
	$u^{(k)}(\textbf x^{(k,\ell,i)}) = u^{(\ell)}(\textbf x^{(k,\ell,i)})$ and thus
	\[
	w^{(k)}|_{\Gamma^{(k,\ell)}} = u^{(k)}|_{\Gamma^{(k,\ell)}} - u^{(\ell)}|_{\Gamma^{(k,\ell)}}.
	\]
	Therefore, we have
	$
	|w^{(k)}|_{H^{1/2}(\Gamma^{(k,\ell)})}^2
	\lesssim |u^{(k)}|_{H^{1/2}(\Gamma^{(k,\ell)})}^2 + | u^{(\ell)}|_{H^{1/2}(\Gamma^{(k,\ell)})}^2 
	$
	and
	$
	|w^{(k)}|_{L_\infty^0(\Gamma^{(k,\ell)})}^2
	\lesssim |u^{(k)}|_{L_\infty^0(\Gamma^{(k,\ell)})}^2 + 
	|u^{(\ell)}|_{L_\infty^0(\Gamma^{(k,\ell)})}^2, 
	$
	which finishes the proof for the Algorithms~A and C.
	
	For Algorithm~B, one obtains
	\[
	\begin{aligned}
	w^{(k)}|_{\Gamma^{(k,\ell)}} &= u^{(k)}|_{\Gamma^{(k,\ell)}} - u^{(\ell)}|_{\Gamma^{(k,\ell)}} 
	\\&+ \sum_{i=1}^2
	\frac{1}{|\mathcal{P}(\textbf x^{(k,\ell,i)})|-1}
	\theta^{(k,\ell,i)} \sum_{j\in\mathcal{P}(\textbf x^{(k,\ell,i)})\backslash\{\ell\}}
	\big( u^{(k)}(\textbf x^{(k,\ell,i)}) - u^{(j)}(\textbf x^{(k,\ell,i)}) \big).
	\end{aligned}
	\]
	Note that $\theta^{(k,\ell,i)}$ behaves like $\max\{0,1-|x-\textbf x^{(k,\ell,i)}|/h^{(k)}\}^p$. So,
	\begin{itemize}
		\item $|\theta^{(k,\ell,i)}|_{H^{1/2}(\Gamma^{(k,\ell)})}^2
		\lesssim |\theta^{(k,\ell,i)}|_{H^1(\Gamma^{(k,\ell)})}\|\theta^{(k,\ell,i)}\|_{L_2(\Gamma^{(k,\ell)})}
		\eqsim 1$, and
		\item $\|\theta^{(k,\ell,i)}\|_{L_\infty(\Gamma^{(k,\ell)})}^2= 1$.
	\end{itemize}
	Assumption~\ref{ass:neighbors} yields $|\mathcal{P}(\textbf x^{(k,\ell,i)})|\eqsim1$. Thus,
	we immediately obtain the desired result also for Algorithm~B.
\end{proof}
The term $\Delta^{(k,\ell,i)}$ contains the differences
$u^{(k)}(\textbf x^{(k,\ell,i)})-u^{(j)}(\textbf x^{(k,\ell,i)})$ for any patch $\Omega^{(j)}$ that shares
the corresponding vertex $\textbf x^{(k,\ell,i)}$ with the patch $\Omega^{(k)}$. The following
Lemma shows that these terms can be estimated from above with differences that only
involve patches sharing an edge.
\begin{lemma}\label{lem:cornerdiff0}
	We have
	\[
	\Delta^{(k,\ell,i)}\lesssim
	\sum_{m \in \mathcal{P}({\normalfont \textbf x}^{(k,\ell,i)})}
	\sum_{n \in \mathcal{P}({\normalfont \textbf x}^{(k,\ell,i)}) \cap \mathcal{N}_\Gamma(m) }
	|u^{(m)}({\normalfont \textbf x}^{(k,\ell,i)})-u^{(n)}({\normalfont \textbf x}^{(k,\ell,i)})|^2,
	\]
	where $\Delta^{(k,\ell,i)}$ is as in Lemma~\ref{lem:bbt}.
\end{lemma}
\begin{proof}
	The sum in $\Delta^{(k,\ell,i)}$ contains contributions of the form
	\[
	|u^{(k)}(\textbf x^{(k,\ell,i)})-u^{(\ell)}(\textbf x^{(k,\ell,i)})|^2.
	\]
	If the patches $\Omega^{(k)}$ and $\Omega^{(\ell)}$ share an edge, we are done.
	Otherwise, 
	there is a sequence $k=:n_0,n_1,\ldots,n_j:=\ell$ such that
	\begin{itemize}
		\item the patches $\Omega^{(n_i)}$ and $\Omega^{(n_{i-1})}$ share an edge,
		thus $n_i\in\mathcal{N}_\Gamma(n_{i-1})$ for all $i=1,\ldots,j$ and
		\item all patches $\Omega^{(n_i)}$ contain the vertex $\textbf x^{(k,\ell,i)}$, thus $n_i \in \mathcal{P}(\textbf x^{(k,\ell,i)})$ for all $i=0,\ldots,j$.
	\end{itemize}
	Thus, we have using Assumption~\ref{ass:neighbors}
	\[
	|u^{(k)}(\textbf x^{(k,\ell,i)})-u^{(\ell)}(\textbf x^{(k,\ell,i)})|^2
	\lesssim \sum_{i=1}^j 
	|u^{(n_{i-1})}(\textbf x^{(k,\ell,i)})-u^{(n_{i})}(\textbf x^{(k,\ell,i)})|^2,
	\]
	which finishes the proof since
	$n_{i-1}\in \mathcal P(\textbf x^{(k,\ell,i)})$ and
	$n_{i}\in \mathcal P(\textbf x^{(k,\ell,i)})\cap \mathcal{N}_\Gamma(n_i)$. 
\end{proof}

\begin{lemma}\label{lem:cornerdiff}
	Let $\Omega^{(k)}$ and $\Omega^{(\ell)}$ be two patches sharing
	the edge $\Gamma^{(k,\ell)}$. Assume that $\Gamma^{(k,\ell)}$ connects
	the two vertices ${\normalfont \textbf x}^{(k,\ell,1)}$ and ${\normalfont \textbf x}^{(k,\ell,2)}$.
	Provided that the integrals over the edge agree, i.e., $\int_{\Gamma^{(k,\ell)}} u^{(k)}(s)-u^{(\ell)}(s)\, \mathrm{d}s = 0$,
	we have
	\[
	\sum_{i=1}^2 |u^{(k)}({\normalfont \textbf x}^{(k,\ell,i)})-u^{(\ell)}({\normalfont \textbf x}^{(k,\ell,i)})|^2
	\lesssim
	\Lambda
	\Big(
	|\mathcal{H}_h^{(k)} u^{(k)}|_{H^1(\Omega^{(k)})}^2
	+ |\mathcal{H}_h^{(\ell)} u^{(\ell)}|_{H^1(\Omega^{(\ell)})}^2
	\Big)
	\]
	for all $u=(u^{(1)},\ldots,u^{(K)})\in V$,
	where $\Lambda := 1+\log p + \max_{j=1,\ldots,K} \log \frac{H^{(j)}}{h^{(j)}}$.
\end{lemma}
\begin{proof}
	Let $\widehat{u}^{(k)}:= u^{(k)}\circ G_k$ and $\widehat{u}^{(\ell)}:= u^{(\ell)}\circ G_\ell$.
	By an unitary transformation (rotation, reflection), we obtain a representation such that
	the pre-image of the joint edge $\Gamma^{(k,\ell)}$ is $(0,1)\times\{0\}$ and that
	the pre-image of $\textbf x^{(k,\ell,1)}$ is $0$. Thus, the assumption on the edge average reads as
	\[
	\int_{(0,1)\times\{0\}} \widehat u^{(k)}(x) - \widehat u^{(\ell)}(x) \mathrm{d}x =0.
	\]
	In the interior, consider the difference of the respective discrete harmonic extensions
	$
	\widehat{\mathcal H}_h^{(k)} \widehat u^{(k)}
	- \widehat{\mathcal H}_h^{(\ell)} \widehat u^{(\ell)}.
	$
	Lemma~\ref{lem:param:point:values} yields
	\[
	|u^{(k)}(\textbf x^{(k,\ell,1)}) - u^{(\ell)}(\textbf x^{(k,\ell,1)})|^2
	= |\widehat u^{(k)}(0) - \widehat u^{(\ell)}(0)|^2
	\lesssim
	\Lambda^{(k,\ell)}
	\| \widehat{\mathcal H}_h^{(k)} \widehat u^{(k)}
	- \widehat{\mathcal H}_h^{(\ell)} \widehat u^{(\ell)} \|_{H^1(\widehat{\Omega})}^2.
	\]
	Using Theorem~1.24 in Ref.~\refcite{Pechstein:2013a} (using the choice
	$\psi(u):=\int_{\widehat\Gamma} u(s) \mathrm{d}s$), we obtain further
	\begin{align*}
		|u^{(k)}(\textbf x^{(k,\ell,1)}) - u^{(\ell)}(\textbf x^{(k,\ell,1)})|^2
		&\lesssim
		\Lambda^{(k,\ell)}
		| \widehat{\mathcal H}_h^{(k)} \widehat u^{(k)}
		- \widehat{\mathcal H}_h^{(\ell)} \widehat u^{(\ell)} |_{H^1(\widehat{\Omega})}^2\\
		&\lesssim
		\Lambda^{(k,\ell)}
		(| \widehat{\mathcal H}_h^{(k)} \widehat u^{(k)}|_{H^1(\widehat{\Omega})}^2
		+| \widehat{\mathcal H}_h^{(\ell)} \widehat u^{(\ell)} |_{H^1(\widehat{\Omega})}^2),
	\end{align*}
	which finishes the proof.
\end{proof}

Using the last three Lemmas and Assumption~\ref{ass:neighbors}, we immediately obtain
\begin{equation}\label{eq:bbt}
\begin{aligned}
&\sum_{k=1}^K \sum_{\ell\in \mathcal{N}_\Gamma(k)} |w^{(k)}|_{H^{1/2}(\Gamma^{(k,\ell)})}^2
\lesssim
\sum_{k=1}^K\sum_{\ell\in \mathcal{N}_\Gamma(k)} \left(  |u^{(k)}|_{H^{1/2}(\Gamma^{(k,\ell)})}^2
+ \Delta^{(k,\ell,1)} + \Delta^{(k,\ell,2)} \right) \\
& \qquad\lesssim
\sum_{k=1}^K \sum_{\ell\in \mathcal{N}_\Gamma(k)}
\left( |u^{(k)}|_{H^{1/2}(\Gamma^{(k,\ell)})}^2
+  \sum_{i=1}^2 |u^{(k)}(\textbf x^{(k,\ell,i)}) - u^{(\ell)}(\textbf x^{(k,\ell,i)})|^2 \right) \\
& \qquad\lesssim
\sum_{k=1}^K\sum_{\ell\in \mathcal{N}_\Gamma(k)}  |u^{(k)}|_{H^{1/2}(\Gamma^{(k,\ell)})}^2
+ \Lambda 
\sum_{k=1}^K |\mathcal{H}_h^{(k)} u^{(k)}|_{H^1(\Omega^{(k)})}^2,
\end{aligned}
\end{equation}
where $\Lambda$ is as in Lemma~\ref{lem:cornerdiff},
and an analogous estimate for the $L_\infty^0(\Gamma^{(k,\ell)})$-seminorms.
Finally, we are able to give a proof of the main theorem.

\begin{proof}\textbf{(of Theorem~\ref{thrm:fin}).}
	Let $u=(u^{(1)},\cdots,u^{(K)})$ with coefficient vector $\underline u$
	be arbitrary but fixed
	and let $w=(w^{(1)},\cdots,w^{(K)})$ with coefficient vector $\underline w$
	be such that
	$
	\underline w = B_D^\top B \underline u
	$.
	Theorem~\ref{thrm:discrharm:ext} yields
	\[
	\| B_D^\top B \underline u \|_S^2
	= \|  \underline w \|_S^2
	= \sum_{k=1}^K |  \mathcal H_h^{(k)} w^{(k)} |_{H^1(\Omega^{(k)})}^2
	\lesssim p \sum_{k=1}^K | w^{(k)} |_{H^{1/2}(\partial\Omega^{(k)})}^2,
	\]
	where $\mathcal H_h^{(k)}$ is the discrete harmonic extension into
	$H^1(\Omega^{(k)})$.
	Lemma~\ref{lem:tearing} yields further
	\[
	\| B_D^\top B \underline u \|_S^2
	\lesssim p \Lambda
	\sum_{k=1}^K\sum_{\ell \in \mathcal N(k)}
	\left( | w^{(k)} |_{H^{1/2}(\Gamma^{(k,\ell)})}^2
	+ | w^{(k)} |_{L_\infty^0(\Gamma^{(k,\ell)})}^2 \right),
	\]
	where $\Lambda$ is as in Lemma~\ref{lem:cornerdiff}.
	Using~\eqref{eq:bbt}, we obtain 
	\begin{align*}
		\| B_D^\top B \underline u \|_S^2
		&\lesssim p \Lambda
		\sum_{k=1}^K
		\sum_{\ell \in \mathcal N_\Gamma(k)}
		\Big( | u^{(k)} |_{H^{1/2}(\Gamma^{(k,\ell)})}^2
		+	|u^{(k)}|_{L_\infty^0(\Gamma^{(k,\ell)})}^2\Big)
		\\&\qquad + p\Lambda^2 \sum_{k=1}^K | \mathcal{H}_h^{(k)} u^{(k)} |_{H^1(\Omega^{(k)})}^2
		.
	\end{align*}
	Lemma~\ref{lem:tearing} and Theorem~\ref{thrm:discrharm:ext} yield
	\begin{align*}
		\| B_D^\top B \underline u \|_S^2 
		&		\lesssim 
		\sum_{k=1}^K
		p \Lambda^2 | \mathcal H_h^{(k)} u^{(k)} |_{H^1(\Omega^{(k)})}^2 
		+ p \Lambda
		| u^{(k)} |_{L_\infty^0(\partial\Omega^{(k)})}^2 .
	\end{align*}
	Lemma \ref{lem:param:point:values} yields
	\[
	\| B_D^\top B \underline u \|_S^2	\lesssim p \Lambda^2
	\sum_{k=1}^K \inf_{c\in \mathbb R}\left(
	| \mathcal H_h^{(k)} u^{(k)} |_{H^1(\Omega^{(k)})}^2
	+  (H^{(k)})^{-2} \| \mathcal H_h^{(k)} (u^{(k)}-c) \|_{L_2(\Omega^{(k)})}^2				
	\right).
	\]
	A standard Poincar{\'e} inequality yields further
	\[
	\| B_D^\top B \underline u \|_S^2
	\lesssim p \Lambda^2
	\sum_{k=1}^K 
	| \mathcal H_h^{(k)} u^{(k)} |_{H^1(\Omega^{(k)})}^2
	= p \Lambda^2\| \underline u \|_S^2
	.
	\]
	The combination of this estimate and~\eqref{eq:MandelDohrmannTezaur} finishes the proof.
\end{proof}

\section{Numerical results}
\label{sec:5}
%
%
%

In this section, we give results from numerical experiments that illustrate
the convergence theory presented in this paper. We consider the Poisson problem
\[
\begin{aligned}
- \Delta u(x,y) & = 2\pi^2 \sin(\pi x)\sin(\pi y) &&\qquad \mbox{for}\quad (x,y)\in\Omega \\
u & = 0 &&\qquad \mbox{on}\quad \partial\Omega,
\end{aligned}
\]
where we consider the two domains
shown in Figure~\ref{fig:computational domains}.
The first domain is a circular ring consisting of 12 patches. Each patch
is parameterized using a NURBS mapping of degree $2$.
The second domain is the Yeti-footprint, where we have decomposed the
$12$ patches of the standard representation into $84$ patches to obtain
a representation with inner vertices as well.

\begin{figure}[h]
	\centering
	\begin{subfigure}{.4\textwidth}
		\centering
		\includegraphics[width=.7\textwidth]{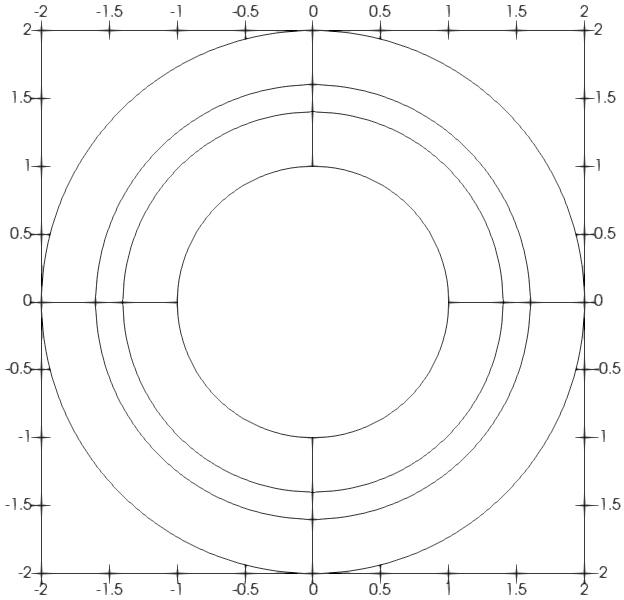}
		{\tiny \caption{Circular ring}\label{subfig:Rings}}
	\end{subfigure} 
	\begin{subfigure}{.4\textwidth}
		\centering
		\includegraphics[width=.7\textwidth]{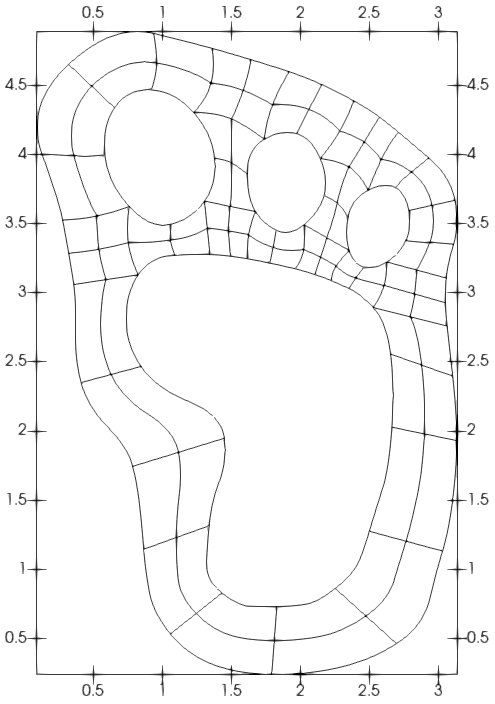}
		{\tiny \caption{Yeti-footprint}
			\label{subfig:Yeti}}
	\end{subfigure} 
	\caption{Computational domains and the decomposition into patches}
	\label{fig:computational domains}
\end{figure}

The numerical experiments are based on the presented IETI-DP approach (Algorithms~A, B 
and C), where the primal degrees of freedom for the corners are implemented by elimination
of the corresponding degrees of freedom as suggested in Section~\ref{sec:3}. 
As patch-local discretization spaces $\widehat V^{(k)}$, we use spline spaces of maximum
smoothness. The discretization of the triple ring on the coarsest grid level ($r=0$) only
consists for each of the patches only of global polynomials, i.e., there are no inner knots.
The same holds for most of the patches of the Yeti-footprint, while the discretization for
the 20 patches that have a non-square shape have one inner knot that connects the midpoints
of the two longer sides of the patch. The finer discretizations ($r=1,2,\ldots$) are obtained by a uniform
refinement such that the grid size behaves like $h\eqsim 2^{-r}$. The newly introduced
knots are single knots, so splines of maximum smoothness are obtained in the interior
of each patch.
The local subproblems are solved using a standard direct solver. The overall problem is solved
using a conjugate gradient solver, preconditioned with the scaled Dirichlet
preconditioner. The starting value has been a vector with random entries in the interval $[-1,1]$.
The iteration has been stopped when the $\ell_2$-norm of the residual
has been reduced by a factor $10^{-6}$ compared to the $\ell_2$-norm of the initial residual.
For all numerical experiments, we present the required number of iterations (it) and the
condition number ($\kappa$) of the preconditioned system $M_{\mathrm{sD}}F$ as estimated
by the conjugate gradient solver.

\begin{table}[t]
\scriptsize
	\newcolumntype{L}[1]{>{\raggedleft\arraybackslash\hspace{-1em}}m{#1}}
	\centering
	\renewcommand{\arraystretch}{1.25}
	\begin{tabular}{l|L{1em}L{1.8em}|L{1em}L{1.8em}|L{1em}L{1.8em}|L{1em}L{1.8em}|L{1em}L{1.8em}|L{1em}L{1.8em}|L{1em}L{1.8em}}
		\toprule
		\multicolumn{1}{l}{r$\;\;\diagdown\;\;$p\hspace{-1.8em}\;}
		& \multicolumn{2}{c|}{2}
		& \multicolumn{2}{c|}{3}
		& \multicolumn{2}{c|}{4}
		& \multicolumn{2}{c|}{5}
		& \multicolumn{2}{c|}{6}
		& \multicolumn{2}{c|}{7}
		& \multicolumn{2}{c}{8} \\
		& it & $\kappa$
		& it & $\kappa$
		& it & $\kappa$
		& it & $\kappa$
		& it & $\kappa$
		& it & $\kappa$
		& it & $\kappa$ \\
		\midrule
		$2$  & $10$ & $5.96$ & $11$ & $6.46$ & $12$ & $6.79$ & $12$ & $7.29$ & $13$ & $ 7.53$ & $14$ & $7.99$ & $15$ & $8.16$ \\
		$3$  & $11$ & $6.69$ & $12$ & $7.30$ & $12$ & $7.81$ & $13$ & $8.28$ & $13$ & $8.64$ & $14$ & $9.03$ & $15$ & $9.32$ \\
		$4$  & $12$ & $7.67$ & $12$ & $8.41$ & $13$ & $9.02$ & $13$ & $9.53$ & $14$ & $9.98$ & $14$ & $10.38$ & $14$ & $10.73$ \\
		$5$  & $12$ & $8.81$ & $13$ & $9.68$ & $14$ & $10.38$ & $14$ & $10.96$ & $14$ & $11.47$ & $15$ & $11.91$ & $15$ & $12.31$ \\
		$6$  & $14$ & $10.14$ & $14$ & $11.13$ & $15$ & $11.91$ & $15$ & $12.57$ & $15$ & $13.13$ & $15$ & $13.63$ & $15$ & $14.07$ \\
		$7$  & $15$ & $11.65$ & $15$ & $12.75$ & $16$ & $13.62$ & $17$ & $14.35$ & $16$ & $14.97$ & $17$ & $15.52$ & $17$ & $16.00$ \\
		$8$  & $15$ & $13.33$ & $17$ & $14.56$ & $17$ & $15.51$ & $17$ & $16.31$ & $17$ & $16.99$ & $17$ &$17.58$ & $17$ &$18.11$ \\
		\bottomrule
	\end{tabular}
	\captionof{table}{Iteration counts (it) and condition numbers $\kappa$; Algorithm A; triple ring
		\label{tab:TripleRing_vertex}}
\end{table}
\begin{table}[t]
\scriptsize
	\newcolumntype{L}[1]{>{\raggedleft\arraybackslash\hspace{-1em}}m{#1}}
	\centering
	\renewcommand{\arraystretch}{1.25}
	\begin{tabular}{l|L{1em}L{1.8em}|L{1em}L{1.8em}|L{1em}L{1.8em}|L{1em}L{1.8em}|L{1em}L{1.8em}|L{1em}L{1.8em}|L{1em}L{1.8em}}
		\toprule
		\multicolumn{1}{l}{r$\;\;\diagdown\;\;$p\hspace{-1.8em}\;}
		& \multicolumn{2}{c|}{2}
		& \multicolumn{2}{c|}{3}
		& \multicolumn{2}{c|}{4}
		& \multicolumn{2}{c|}{5}
		& \multicolumn{2}{c|}{6}
		& \multicolumn{2}{c|}{7}
		& \multicolumn{2}{c}{8} \\
		& it & $\kappa$
		& it & $\kappa$
		& it & $\kappa$
		& it & $\kappa$
		& it & $\kappa$
		& it & $\kappa$
		& it & $\kappa$ \\
		\midrule
		$2$  & $14$ & $15.43$ & $13$ &  $15.69$ & $13$ & $15.56$ & $15$ & $16.32$ & $15$ & $16.33$ & $16$ &  $17.01$ & $16$ &  $17.04$ \\
		$3$  & $15$ & $18.32$ & $15$ &  $18.51$ & $15$ & $18.93$ & $16$ & $19.45$ & $17$ & $19.81$ & $17$ & $20.30$ & $18$ & $20.63$ \\
		$4$  & $16$ &  $21.52$ & $16$ & $21.77$ & $16$ & $22.24$ & $16$ & $22.73$ & $16$ & $23.19$ & $17$ & $23.63$ & $18$ & $24.02$ \\
		$5$  & $17$ & $24.92$ & $17$ & $25.17$ & $18$ & $25.67$ & $17$ & $26.18$ & $18$ & $26.67$ & $18$ & $27.12$ & $18$ & $27.54$ \\
		$6$  & $18$ & $28.47$ & $18$ & $28.73$ & $19$ & $29.25$ & $19$ & $29.78$ & $19$ & $30.28$ & $19$ & $30.75$ & $21$ & $31.19$ \\
		$7$  & $19$ & $32.16$ & $19$ & $32.43$ & $21$ & $32.97$ & $22$ & $33.52$ & $21$ & $34.05$ & $22$ & $34.54$ & $21$ & $34.99$ \\
		$8$  & $21$ & $36.01$ & $22$ & $36.28$ & $22$ & $36.84$ & $22$ & $37.42$ & $23$ & $37.96$ & $23$ & $38.47$ & $24$ & $38.95$ \\
		\bottomrule
	\end{tabular}
	\captionof{table}{Iteration counts (it) and condition numbers $\kappa$; Algorithm B; triple ring
		\label{tab:TripleRing_edge}}
\end{table}
\begin{table}[t]
\scriptsize
	\newcolumntype{L}[1]{>{\raggedleft\arraybackslash\hspace{-1em}}m{#1}}
	\centering
	\renewcommand{\arraystretch}{1.25}
	\begin{tabular}{l|L{1em}L{1.8em}|L{1em}L{1.8em}|L{1em}L{1.8em}|L{1em}L{1.8em}|L{1em}L{1.8em}|L{1em}L{1.8em}|L{1em}L{1.8em}}
		\toprule
		\multicolumn{1}{l}{r$\;\;\diagdown\;\;$p\hspace{-1.8em}\;}
		& \multicolumn{2}{c|}{2}
		& \multicolumn{2}{c|}{3}
		& \multicolumn{2}{c|}{4}
		& \multicolumn{2}{c|}{5}
		& \multicolumn{2}{c|}{6}
		& \multicolumn{2}{c|}{7}
		& \multicolumn{2}{c}{8} \\
		& it & $\kappa$
		& it & $\kappa$
		& it & $\kappa$
		& it & $\kappa$
		& it & $\kappa$
		& it & $\kappa$
		& it & $\kappa$ \\
		\midrule
		$2$  & $9$ & $2.18$ & $9$ & $2.33$ & $10$ & $2.55$ & $10$ & $2.68$ & $11$ & $2.91$ & $11$ & $3.01$ & $12$ & $3.24$ \\
		$3$  & $10$ & $2.46$ & $10$ & $2.74$ & $10$ & $3.00$ & $11$ & $3.21$ & $11$ & $3.42$ & $12$ & $3.58$ & $12$ & $3.77$ \\
		$4$  & $10$ & $2.92$ & $11$ & $3.29$ & $11$ & $3.59$ & $11$ & $3.85$ & $12$ & $4.07$ & $12$ & $4.26$ & $12$ & $4.44$ \\
		$5$  & $11$ & $3.47$ & $11$ & $3.93$ & $12$ & $4.27$ & $12$ & $4.55$ & $13$ & $4.80$ & $13$ & $5.01$ & $13$ & $5.20$ \\
		$6$  & $12$ & $4.16$ & $12$ & $4.64$ & $13$ & $5.01$ & $13$ & $5.32$ & $13$ & $5.59$ & $14$ & $5.82$ & $14$ & $6.03$ \\
		$7$  & $13$ & $4.89$ & $14$ & $5.42$ & $14$ & $5.82$ & $15$ & $6.16$ & $15$ & $6.44$ & $15$ & $6.69$ & $16$ & $6.91$ \\
		$8$  & $14$ & $5.69$ & $15$ & $6.25$ & $15$ & $6.69$ & $16$ & $7.05$ & $16$ & $7.35$ & $16$ & $7.61$ & $16$ & $7.84$ \\
		\bottomrule
	\end{tabular}
	\captionof{table}{Iteration counts (it) and condition numbers $\kappa$; Algorithm C; triple ring
		\label{tab:TripleRing_vertexEdge}}
\end{table}

In the Tables~\ref{tab:TripleRing_vertex}, \ref{tab:TripleRing_edge}
and~\ref{tab:TripleRing_vertexEdge}, we present the numerical experiments
for the triple ring (Figure~\ref{subfig:Rings}).
In Table~\ref{tab:TripleRing_vertex}, we provide the results for the case
that only the vertex values are chosen as primal degrees of freedom (Algorithm~A).
Here, we observe that the condition number grows like $\log^2 H/h$ as predicted
by the theory. The growth in the spline degree seems to be like $\sqrt p$ or even
slower, while the theory predicts $p \log^2 p$.
In Table~\ref{tab:TripleRing_edge}, we consider the case that only the averages
of the function values on the edges are chosen as primal degrees of freedom (Algorithm~B).
Here, the condition number seems to be larger than that of Algorithm~A. Here, the
growth of the condition number in the grid size looks like $\log^2 H/h$. The growth
in the spline degree seems to be linear.
Table~\ref{tab:TripleRing_vertexEdge} shows the results for the case that both
kinds of primal degrees of freedom are combined (Algorithm~C). As expected, the
obtained condition numbers are smaller than those obtained with Algorithms A and B.
We observe the same behavior as for Algorithm~A.

\begin{table}[t]
\scriptsize
	\newcolumntype{L}[1]{>{\raggedleft\arraybackslash\hspace{-1em}}m{#1}}
	\centering
	\renewcommand{\arraystretch}{1.25}
	\begin{tabular}{l|L{1em}L{1.8em}|L{1em}L{1.8em}|L{1em}L{1.8em}|L{1em}L{1.8em}|L{1em}L{1.8em}|L{1em}L{1.8em}|L{1em}L{1.8em}}
		\toprule
		\multicolumn{1}{l}{r$\;\;\diagdown\;\;$p\hspace{-1.8em}\;}
		& \multicolumn{2}{c|}{2}
		& \multicolumn{2}{c|}{3}
		& \multicolumn{2}{c|}{4}
		& \multicolumn{2}{c|}{5}
		& \multicolumn{2}{c|}{6}
		& \multicolumn{2}{c|}{7}
		& \multicolumn{2}{c}{8}  \\
		& it & $\kappa$
		& it & $\kappa$
		& it & $\kappa$
		& it & $\kappa$
		& it & $\kappa$
		& it & $\kappa$
		& it & $\kappa$  \\
		\midrule
		$1$  & $12$ &  $3.02$ & $14$ &  $4.06$ & $15$ &  $4.72$ & $17$ &  $5.60$ & $18$ &  $6.13$ & $20$ &  $6.88$ & $22$ &  $7.33$ \\
		$2$  & $14$ &  $4.64$ & $16$ &  $5.76$ & $17$ &  $6.64$ & $18$ &  $7.49$ & $20$ &  $8.15$ & $21$ &  $8.86$ & $22$ &  $9.39$ \\
		$3$  & $16$ &  $6.71$ & $19$ &  $8.03$ & $20$ &  $9.11$ & $20$ &  $10.04$ & $21$ & $10.83$ & $23$ & $11.57$ & $24$ & $12.20$ \\
		$4$  & $20$ &  $9.22$ & $22$ & $10.77$ & $22$ & $12.03$ & $23$ & $13.09$ & $23$ & $14.01$ & $24$ & $14.84$ & $26$ & $15.57$ \\
		$5$  & $22$ & $12.17$ & $23$ & $13.96$ & $25$ & $15.40$ & $25$ & $16.60$ & $26$ & $17.64$ & $27$ & $18.59$ & $27$ & $19.39$ \\
		$6$  & $25$ & $15.58$ & $26$ & $17.59$ & $26$ & $19.20$ & $27$ & $20.55$ & $28$ & $21.72$ & $28$ & $22.73$ & $29$ & $23.65$ \\
		$7$  & $27$ & $19.42$ & $28$ & $21.69$ & $28$ & $23.48$ & $29$ & $24.94$ & $31$ & $26.26$ & $30$ & $27.35$ & $32$ & $28.79$ \\
		\bottomrule
	\end{tabular}
	\captionof{table}{Iteration counts (it) and condition numbers $\kappa$; Algorithm A; Yeti-footprint
		\label{tab:Yeti_vertex}}
\end{table}
\begin{table}[t]
\scriptsize
	\newcolumntype{L}[1]{>{\raggedleft\arraybackslash\hspace{-1em}}m{#1}}
	\centering
	\renewcommand{\arraystretch}{1.25}
	\begin{tabular}{l|L{1em}L{1.8em}|L{1em}L{1.8em}|L{1em}L{1.8em}|L{1em}L{1.8em}|L{1em}L{1.8em}|L{1em}L{1.8em}|L{1em}L{1.8em}}
		\toprule
		\multicolumn{1}{l}{r$\;\;\diagdown\;\;$p\hspace{-1.8em}\;}
		& \multicolumn{2}{c|}{2}
		& \multicolumn{2}{c|}{3}
		& \multicolumn{2}{c|}{4}
		& \multicolumn{2}{c|}{5}
		& \multicolumn{2}{c|}{6}
		& \multicolumn{2}{c|}{7}
		& \multicolumn{2}{c}{8} \\
		& it & $\kappa$
		& it & $\kappa$
		& it & $\kappa$
		& it & $\kappa$
		& it & $\kappa$
		& it & $\kappa$
		& it & $\kappa$ \\
		\midrule
		$1$  &  $8$ &  $1.47$ &  $9$ &  $1.64$ &  $10$ &  $1.79$ &  $11$ &  $1.94$ &  $12$ &  $2.08$ & $13$ &  $2.22$ & $14$ &  $2.34$ \\
		$2$  &  $10$ &  $2.03$ &  $11$ &  $2.25$ & $12$ &  $2.46$ & $13$ &  $2.65$ & $14$ &  $2.82$ & $15$ &  $2.94$ & $16$ &  $3.13$ \\
		$3$  & $13$ &  $2.77$ & $14$ &  $3.04$ & $15$ &  $3.31$ & $15$ &  $3.55$ & $16$ &  $3.76$ & $17$ &  $3.95$ & $18$ &  $4.13$ \\
		$4$  & $15$ &  $3.70$ & $16$ &  $4.03$ & $17$ &  $4.33$ & $18$ &  $4.62$ & $18$ &  $4.87$ & $20$ &  $5.09$ & $21$ &  $5.31$ \\
		$5$  & $18$ &  $4.80$ & $19$ &  $5.19$ & $19$ &  $5.55$ & $20$ &  $5.87$ & $21$ &  $6.15$ & $21$ &  $6.41$ & $23$ &  $6.65$ \\
		$6$  & $20$ &  $6.08$ & $21$ &  $6.52$ & $22$ &  $6.93$ & $22$ &  $7.29$ & $23$ &  $7.62$ & $24$ &  $7.91$ & $24$ &  $8.17$ \\
		$7$  & $22$ &  $7.53$ & $23$ &  $8.03$ & $23$ &  $8.49$ & $24$ &  $8.90$ & $25$ &  $9.26$ & $25$ &  $9.58$ & $26$ &  $9.87$ \\
		\bottomrule
	\end{tabular}
	\captionof{table}{Iteration counts (it) and condition numbers $\kappa$; Algorithm B; Yeti-footprint
		\label{tab:Yeti_edge}}
\end{table}
\begin{table}[t]
\scriptsize
	\centering
	\renewcommand{\arraystretch}{1.25}
	\begin{tabular}{l|rr|rr|rr|rr|rr|rr|rr}
		\toprule
		\multicolumn{1}{l}{r$\;\;\diagdown\;\;$p\hspace{-1.8em}\;}
		& \multicolumn{2}{c|}{2}
		& \multicolumn{2}{c|}{3}
		& \multicolumn{2}{c|}{4}
		& \multicolumn{2}{c|}{5}
		& \multicolumn{2}{c|}{6}
		& \multicolumn{2}{c|}{7}
		& \multicolumn{2}{c}{8} \\
		& it & $\kappa$
		& it & $\kappa$
		& it & $\kappa$
		& it & $\kappa$
		& it & $\kappa$
		& it & $\kappa$
		& it & $\kappa$ \\
		\midrule
		$1$  &  $6$ &  $1.22$ &  $7$ &  $1.38$ &  $8$ &  $1.50$ &  $9$ &  $1.65$ &  $10$ &  $1.77$ & $11$ &  $1.88$ & $12$ &  $1.97$ \\
		$2$  &  $8$ &  $1.43$ &  $9$ &  $1.61$ & $10$ &  $1.78$ & $11$ &  $1.95$ & $12$ &  $2.10$ & $13$ &  $2.25$ & $14$ &  $2.39$ \\
		$3$  & $10$ &  $1.84$ & $11$ &  $2.11$ & $12$ &  $2.34$ & $13$ &  $2.56$ & $13$ &  $2.74$ & $14$ &  $2.91$ & $16$ &  $3.09$ \\
		$4$  & $12$ &  $2.48$ & $13$ &  $2.82$ & $14$ &  $3.12$ & $15$ &  $3.38$ & $16$ &  $3.62$ & $17$ &  $3.83$ & $18$ &  $4.04$ \\
		$5$  & $14$ &  $3.32$ & $15$ &  $3.75$ & $16$ &  $4.13$ & $17$ &  $4.45$ & $18$ &  $4.73$ & $19$ &  $4.99$ & $20$ &  $5.23$ \\
		$6$  & $17$ &  $4.41$ & $18$ &  $4.93$ & $18$ &  $5.37$ & $19$ &  $5.76$ & $20$ &  $6.10$ & $21$ &  $6.40$ & $22$ &  $6.69$ \\
		$7$  & $19$ &  $5.73$ & $20$ &  $6.34$ & $21$ &  $6.86$ & $21$ &  $7.31$ & $22$ &  $7.70$ & $23$ &  $8.05$ & $24$ &  $8.37$ \\
		\bottomrule
	\end{tabular}
	\captionof{table}{Iteration counts (it) and condition numbers $\kappa$; Algorithm C; Yeti-footprint
		\label{tab:Yeti_vertexEdge}}
\end{table}

In the Tables~\ref{tab:Yeti_vertex}, \ref{tab:Yeti_edge}
and~\ref{tab:Yeti_vertexEdge}, we present the numerical experiments
for the Yeti-footprint (Figure~\ref{subfig:Yeti}). Here, we observe in all
cases that the growth of the condition number in the grid size is like
$\log^2 H/h$ and the growth in the spline degree is like $\sqrt p$ or slightly
smaller. Opposite to the results for the triple ring, we observe that Algorithm~B performs
better than Algorithm~A. Again, Algorithm~C yields the smallest condition numbers, see Table~\ref{tab:Yeti_vertexEdge}.

\section{Conclusions}
\label{sec:6}
In the paper, we have extended the known convergence analysis for IETI-DP methods
such that it also covers the case that only edge-averages are used as primal degrees
of freedom. Moreover, we provide estimates for the condition number that are explicit
both in the grid sizes and in the spline degree. The main ingredient for that analysis
is an estimate for the discrete harmonic extension, cf. Section~\ref{sec:4:1}. The
dependence of the condition number on the grid size is the same as in any other standard
analysis. In the spline degree, the bound on the convergence number depends like
$p \log^2 p$ on the spline degree $p$. Most numerical experiments indicate that
the true growth might be smaller, but one of the numerical experiments has shown a
growth similar to the upper bound from the convergence theory.
Since in Isogeometric Analysis, only moderate
values of $p$ are of interest, the corresponding dependence seems to be moderate.

Note that the condition number of the stiffness matrix $\kappa(A)$ grows exponentially in the
spline degree. This means that the condition number of the preconditioned IETI-DP
system grows only logarithmic in $\kappa(A)$, if the spline degree $p$ is increased.
The same is observed for the grid size, since $\kappa(A)\eqsim h^{-2}$ and
the condition number of the preconditioned IETI-DP system grows (poly-)logarithmic in $h^{-1}$
or, equivalently, in $\kappa(A)$.

The extension of the presented method to problems with varying diffusion coefficients
seems to be straight-forward. Our analysis for Algorithm~B also allows the extension of the
IETI-DP solvers (and their analysis) to discretizations with non-matching patches, that
are of interest for moving domains, like for the analysis of electrical motors.

\section*{Acknowledgments}
The first author was supported by the Austrian Science Fund (FWF): S117 and 
W1214-04, while the second author has received support by the Austrian Science
Fund (FWF): P31048.
Moreover, the authors thank Ulrich Langer for fruitful discussions and help
with the study of existing literature.

%
%
\bibliography{Literature.bib}

\end{document}